\documentclass{article}
\usepackage[T1]{fontenc}         
\usepackage{graphicx}
\usepackage{caption}
\usepackage{subcaption} 
\usepackage{float}   
\usepackage{fancybox}		  
\usepackage{makeidx} 
\usepackage{verbatim}
\usepackage{listings} % package for a colored Matlab code            
\usepackage{fancyvrb}
\usepackage{amsfonts}
\usepackage{color}
\usepackage{colortbl}
\usepackage{amsmath}
\usepackage{color}

\usepackage{epsfig}
\usepackage{verbatim}
\usepackage{listings}

\usepackage{amssymb}
\usepackage{amsbsy}
\usepackage{amsmath}
\usepackage{enumerate}
\usepackage{stmaryrd}
\usepackage{mathtools}
\usepackage{mathrsfs}
\usepackage{dsfont}

\usepackage{amsthm}
\usepackage{tikz}
\usepackage[margin=1in]{geometry}
\usepackage{fancyhdr}
\usepackage{amsthm}

\DeclareMathOperator*{\argmax}{argmax}

\newtheorem{prop}{Proposition}[section]
\newtheorem{theorem}{Theorem}[section]
\newtheorem{lemma}{Lemma}[section]
\newtheorem{rem}{Remark}[section]%[P\arabic{prob}]
\newtheorem{definition}{Definition}[section]

\renewcommand{\leq}{\leqslant}
\renewcommand{\geq}{\geqslant}

\newcommand{\RR}{\mathbb{R}}
\newcommand{\R}{\mathbb{R}}
\newcommand{\N}{\mathbb{N}}
\newcommand{\UM}{U_M}
\newcommand{\Mc}{\mathcal{M}_\mathrm{c}}
\newcommand{\Scl}{\mathcal{S}_\mathrm{cl}}
\newcommand{\M}{\mathcal{M}}
\newcommand{\elts}{\{1,\ldots,N\}}
\renewcommand{\P}{\mathcal{P}}
\renewcommand{\div}{\mathrm{div}}
\newcommand{\UMkin}{\mathcal{U}_M}
\newcommand{\Wg}{\mathcal{W}_g}

\newcommand{\Rkin}{\mathcal{R}}
\newcommand{\V}{\mathcal{V}}
\newcommand{\Skin}{\mathcal{S}}
\newcommand{\supp}{\mathrm{supp}}
\newcommand{\bx}{\bar{x}}

\newcommand{\DDt}{\frac{d }{d t}}
\newcommand{\one}{\mathds{1}}
\newcommand{\Lip}{\mathrm{Lip}}

\begin{document} 

\title{Sparse control of Hegselmann-Krause models: Black hole and declustering}

\author{ Benedetto Piccoli\thanks{Department of Mathematical Sciences, Rutgers University - Camden, Camden, NJ. {\tt piccoli@camden.rutgers.edu}},
 Nastassia Pouradier Duteil\thanks{C\'er\'emade, Universit\'e Paris-Dauphine, Paris, France. {\tt n.pouradier@gmail.com}}, Emmanuel Tr\'elat\thanks{Sorbonne Universit\'e, Universit\'e Paris-Diderot SPC, CNRS, Inria, Laboratoire Jacques-Louis Lions, \'equipe CAGE, F-75005 Paris. {\tt emmanuel.trelat@upmc.fr}}}

\date{}

\maketitle

\begin{abstract}
This paper elaborates control strategies to prevent clustering effects in opinion formation models. 
This is the exact opposite of numerous situations encountered in the literature where, on the contrary, one seeks controls promoting consensus.
% Interestingly, maximizing the classical variance does not prevent clustering.
 In order to promote declustering, instead of using the classical variance that does not capture well the phenomenon of dispersion, we introduce an entropy-type functional that is adapted to measuring pairwise distances between agents.
We then focus on a Hegselmann-Krause-type system and design declustering sparse controls both in finite-dimensional and kinetic models. 
We provide general conditions
characterizing whether clustering can be avoided as function
of the initial data. Such results include the description of black
holes (where complete collapse to consensus is not avoidable),
safety zones (where the control can keep the system far from
clustering), basins of attraction (attractive zones around the
clustering set) and collapse prevention (when convergence to the clustering set can be avoided).

\paragraph{Keywords:} Active particles, collective behavior, control, kinetic model, declustering, black swan.

\paragraph{AMS Subject Classifications:} 93C15, 93C20, 91D10, 35B36, 34H05.
\end{abstract}

\section*{Introduction}

The term ``black swan'' was first used by Nassim Nicholas Taleb in 2007 in his book \textit{The Black Swan: The Impact of the Highly Improbable} \cite{Taleb10}, in which he focuses on the extreme impact of rare and unpredictable events. The ``black swan theory'' was since then developed to describe events that are extremely rare, have a massive impact and are retrospectively predictable. 
One of the groundbreaking ideas of this recent theory is the fact that human behavior remains unpredictable. By focusing on what is known and probable, scientists tend to be surprised by major unexpected events. Taleb's philosophy requires one to accept the fact that there will always remain unknown factors - hence, one cannot make future predictions based only on the assumption of a population's rational behavior.

Bellomo et al. \cite{BHT12} have built upon this theory, applying it to the context of social competition that can lead to extreme conflicts. 
Their work is based upon the fact that individual behavior, whether rational or irrational, contributes in a nonlinear fashion to the global group behavior.
Then, even among an initially well distributed population, local social interactions can lead to unexpected outcome. 
%It is said that the global patterns \textit{emerge} from the local interactions.

Social Dynamics models are particularly suited to describe these kinds of phenomena, as they focus on understanding how \textit{self-organization} emerges from interactions of individual ``agents'', or ``active particles'' \cite{ACMPPRT17}. 
%unwanted clustering (of wealth, opinions, etc.).
The study of collective behavior emerging from local interactions is actually of great interest to a mixed community of mathematicians, biologists, sociologists, economists and engineers. 
These models are indeed applicable to a wide variety of fields. In biology, they are used to understand the behavior of large animal groups 
\cite{BCC08, CCR11, CCGPSSV10, CKFL05, PE99, VCABC95, YEECBKM09}.
Engineering applications involve robot formation and satellite synchronization \cite{B09, JLM03, MMLK10, PEG09, SBS10, SPL2007}.
Models also apply to socio-economic problems such as population dynamics, opinion formation and market evolution \cite{BHT12, HK02, Kirman2000, MBE08, S82, Taleb10}.
As pointed out in \cite{VRSM12}, individual behavior, especially human, is often irrational: instead of making strategic decisions, individuals tend to imitate social neighbors. This behavior leads to clustering of opinions or even consensus (agreement of all state variables). Many models reproduce this phenomenon. In \cite{VRSM12}, this is modeled in a game-theoretic set-up, where agents play coordination games to improve their individual payoff. In the Voter model, agents imitate the action of a randomly selected counterpart \cite{HL75}. In the Hegselmann-Krause (HK) bounded-confidence model, agents imitate others' behavior only if they are within a certain ``confidence'' radius \cite{HK02}. In a competing approach, based on the so-called ``topological'' distance, agents imitate a given number of closest neighbors \cite{BCC08}. Another variation of the HK model consists of noticing that heterophilious dynamics enhance consensus \cite{MoTa14}.
%Second-order models, such as the well-studied Cucker-Smale one \cite{CS07}, may lead to alignment, \textit{i.e.} agreement in the second variable under suitable conditions on the interaction function \cite{CFPT13, HHK10}. 

Multi-agent systems can be described from a microscopic point of view, by considering a system of coupled (often nonlinear) ODE's \cite{CFPT13,CFPT15, CS07, L13}. However, as the dimension of the system increases, studying and simulating it becomes a harder challenge, a phenomenon known as the \textit{curse of dimensionality}. When the number of agents tends to infinity, one can take the \textit{mean-field limit} of the system resulting in a kinetic model, where the population is described by a density measure, and its evolution is given by a unique PDE. The mean-field limits of the Hegselmann-Krause, Vicsek and Cucker-Smale models were respectively derived in \cite{CFT08, DM08, HT08}. Since then, kinetic formulations of Social Dynamics models have been the focus of many more works, see for example \cite{BP14, CCR11, CFTV10, DM08_2, FS14, LW11, PRT15}.

Self-organization 
has thus been extensively studied, especially focusing on the emergence of patterns such as consensus or alignment that arise from inherent properties of certain dynamics.
When consensus is not reached by the system, it is natural to ask whether it can be achieved by controlling it. Such control problems have been investigated in finite dimensional systems  \cite{CFPT13, CFPT15, CPRT17, L13} and in kinetic models \cite{CPRT17-2, PPS15}.
Applications involve rendez-vous problems in robotics, and flock formation in animal crowd behavior.  
However, as seen in \cite{B09}, the states of consensus or clustering are not always desirable as they can be seen as a manifestation of Black Swan effects.
Therefore we choose to study the opposite problem: given dynamics naturally leading to consensus, we aim to, at the contrary, control the system to avoid consensus and clustering, \textit{i.e.}, to keep the agents as far from one another as possible. Possible motivations include keeping a market from collapsing or a crowd from converging to a localized dense conformation.
While typical control problems applied to social dynamics models aim to steer the system to consensus, which is a natural feature of the dynamics, here instead we aim to drive the system against its natural behavior. The key is then to understand the interplay between the internal driving force of the system and the external applied control. It is natural to expect that the feasibility of the system will depend on the allowed strength of the control and on the nature of the interaction function. Indeed, the main results of this paper will highlight in particular the existence of internal attraction so strong that no control can act on the system: we will refer to such cases as ``black holes''.

We study a first-order opinion formation model with a positive interaction function $a(\cdot)$ and control the system via an additive term:
$$
\dot{x}_i(t) = \frac{1}{N}\sum_{j\neq i} a(\|x_i(t)-x_j(t)\|)(x_j(t)-x_i(t)) + u_i(t), \quad i\in\elts.
$$
In its mean-field limit, \textit{i.e.} when $N$ tends to infinity, we study the kinetic model
$$
\partial_t\mu + \div\left( \left(\int_{\R^d} a(\|x-y\|)(y-x)d\mu(y)+\chi_\omega u\right)\mu \right) = 0.
$$
All terms and assumptions for both finite-dimensional and kinetic formulations are rigorously defined later in Sections \ref{Sec:finite} and \ref{Sec:kinetic}.
The control strategies used to steer the system away from clustering are simple explicit feedbacks depending on the state of the system at the current time. 
In the microscopic model, the control can be seen as an exterior force acting on the system to separate the agents from one another. For realistic purposes, we assume that we can only act on the system with finite strength, and we impose a constraint on the $\ell^1-\ell^2$ norm of the control: $\sum_i \|u_i\|\leq M$ for some $M>0$. This constraint is known to promote sparsity (see \cite{CFPT13}), so that
we promote controls that act on fewer agents at a time.
In the macroscopic equation, 
we consider the class of controls $\chi_\omega u$, where $u\in L^\infty(\R^+\times\R^d)$ and for all $t\geq 0$, $\omega(t)$ is a measurable subset of $\R^d$ (and $\chi$ is the indicator function). The control is thus constrained in two ways, as for the microscopic case.
The amplitude $u$ is constrained by the condition $\|u\|_{L^\infty(\R^+\times \R^d)}\leq M$. 
The spacial sparsity of the control is ensured by the condition $\int_{\omega(t)} dx\leq c$ for some $c>0$. This particular choice of sparsity constraint means that we only allow the control to act on a given area of space. Instead, one could consider constraining the portion of the population being controlled, as done for instance in \cite{PRT15} for the control to flocking of the kinetic Cucker-Smale model.

We show that to prevent the formation of consensus, the controls may be chosen to maximize the derivative of the variance of the system. However this strategy is not enough to achieve the stronger requirement of avoiding any degree of clustering. We show that the state of clustering may be achieved thanks to a different entropy-type functional  
 that measures the \emph{dispersion} of the system.
Depending on the behavior of the interaction function $a(\cdot)$, several situations may arise. If $\lim_{s\rightarrow 0} sa(s) = + \infty $, there exists a ``\textit{black hole}'' region, in which no control can prevent the system from converging.
In contrast, if $\lim_{s\to 0} sa(s) = 0$, collapse to consensus can always be avoided.  
Far from the consensus manifold, we also observe two scenarios. 
If $\lim_{s\rightarrow +\infty} sa(s) = 0 $, there exists a ``\textit{safety zone}'' in which the control can always keep the system far from consensus. This safety zone does not exist if $\lim_{s\rightarrow +\infty} sa(s)=+\infty$, as the system converges to a ``\textit{basin of attraction}''.

We summarize these results in Table \ref{table:summary}, giving criteria depending on $\alpha$ and $\bar{s}$, where $\alpha:=\lim_{s\to\bar{s}}s a(s)$. 
%The limit of $sa(s)$ when $s\to 0$ determines the existence of a black hole near the consensus manifold, which is a subset of $\RR^{dN}$ containing the consensus manifold that no control allows to escape from. 
%On the other hand, the limit at infinity determines the existence of a safety zone far from the consensus manifold. 

 \begin{table}[H] 
\centering
 \begin{tabular}{|c|p{6cm}|p{6cm}|}
 \hline
 & {\centering $\bar{s} = 0$} & {\centering $\bar{s} = +\infty$}\\
 \hline
 $\alpha = 0$ & There exists a \textbf{collapse prevention} control strategy  & There exists a \textbf{safety region} far from consensus  \\
 \hline
 $\alpha = +\infty $ & There exists a \textbf{black hole} (no strategy can avoid consensus for certain initial configurations)& There exists a \textbf{basin of attraction } (no safety zone far from consensus) \\
 \hline
 \end{tabular}
 \caption{Four different configurations determined by $\alpha=\lim_{s\to\bar{s}}s a(s)$}\label{table:summary}
\end{table}

The paper is divided into three parts: in the first one, we consider a microscopic description of the system adapted from the Hegselmann-Krause opinion formation model. Secondly, we study the kinetic version of this model by taking the mean-field limit of the system, and we adequately extend the results of the microscopic description to the kinetic setting. Lastly, we provide numerical simulations illustrating the four cases presented in Table \ref{table:summary}.

\section{Microscopic model (finite-dimension)} \label{Sec:finite}

\subsection{Generalized entropy functional for declustering control} \label{Sec:Entropy}

Consider the 
general class of first-order 
differential 
systems:
\begin{equation}\label{eq:gensyst0}
\dot x_i = f_i(x), \quad i\in \elts
%+ u_i,\qquad i=1,\ldots,N,\quad x_i(t)\in\mathbb{R}^d,
\end{equation}
where for each $i\in\elts$, $x_i(t)\in\mathbb{R}^d$. The dynamics are given by the functions $f_i\in C^1((\RR^d)^n,\RR^d)$.
%Each agent is controlled via the function $u_i:\RR\rightarrow\RR^d$. 
%Further assumptions on the control set will be made in the following sections. Although we allow the control to act on each agent $i$, we will later design some strategies that naturally lead to a component-sparse control, \textit{i.e.} that acts only on one agent at a time.  
% (they do not necessarily depend on $\Vert x_i-x_j\Vert$). 

The purpose of this work is to study the collective behavior of the system, focusing on patterns such as \textit{consensus} and \textit{clustering}. More specifically, we aim to design feedback control strategies to prevent the system from reaching those states. We first provide the general definitions that will be used hereafter.

\begin{definition}
The state characterized by  $x_1=...=x_N$ is referred to as \textbf{consensus}.
We denote by $\Mc$ the \textbf{consensus manifold} defined by:
\begin{equation}
\Mc := \{(x_i)_{i\in\{1,\ldots,N\}}\; | \; \forall (j,k)\in\{1,\ldots,N\}^2, \; x_j=x_k\}.
\end{equation}
\end{definition}

\begin{rem}
If the dynamics satisfy $f_i(x)=0$ for every $i\in\elts$ and for all $x\in\Mc$, 
the consensus state is an equilibrium.
% for the Hegselmann-Krause opinion dynamics without control (\textit{i.e.} $u\equiv 0$). For this reason, system \eqref{eq:Krause1} is sometimes referred to as \textbf{consensus dynamics}.
\end{rem}

Notice that if at least two agents have different states, for instance if $x_i\neq x_j$ for some $i,j\in\{1,\ldots,N\}^2$, then the system is not in consensus. 
\begin{definition}\label{def:dispersion}
The system is said to \textbf{avoid consensus} if there exist $(i,j)\in\{1,\ldots,N\}^2$ such that $x_i\neq x_j$.
\end{definition}

However, avoiding consensus might still leave the system in the critical state where several agents have the same state variable, which might be unwanted in some real-life situations. For instance, if each $x_i$ represents an investor's decision, consensus might lead to a market crash. Whether or not the system is exactly in consensus state has little impact on the outcome: if one investor thinks differently than the mass (e.g. $x_j\neq x_1 = \ldots = x_{j-1} = x_{j+1} = \ldots = x_N$), it might not be enough to prevent a market collapse. 
With such applications in mind, we define clustering as follows:

\begin{definition}\label{def:clustering}
The system is said to be \textbf{clustered} if there exist $(i,j)\in\{1,\ldots,N\}^2$ such that $x_i=x_j$. We denote by $\Scl$ the \textbf{clustering set}, defined by:
\begin{equation}
\Scl := \{(x_i)_{i\in\{1,\ldots,N\}}\; | \; \exists (j,k)\in\{1,\ldots,N\}^2 \text{ s.t. } x_j=x_k\}.
\end{equation}
\end{definition}

Since we will focus on the avoidance of clustering, we characterize it as follows:

\begin{definition}
We say that the system is \textbf{fully declustered} if there exists $\epsilon>0$ such that for all $(i,j)\in\{1,\ldots, N\}^2$, $\|x_i-x_j\|\geq\epsilon$.
\end{definition}

Notice that the condition of avoiding consensus is weaker than the condition of declustering. 
The system is said to avoid consensus if it is not in a neighborhood of the consensus manifold. 
More constraining, the condition of declustering is satisfied if and only if the system is outside of a neighborhood of a larger manifold, that we refer to as the \textbf{clustering} set.

The consensus manifold $\Mc$ is thus contained in the clustering set $\Scl$. More specifically, $\Mc$ is a $d-$dimensional manifold embedded in $(\RR^d)^N$, while $\Scl$ is a stratified set in the sense of Whitney (see Figure \ref{fig:consensusmanifold}). We recall that
a set $E\subset \RR^n$ is called \textbf{stratified} in the sense of Whitney if there exists a countable (locally finite) collection of pairwise disjoint manifolds $(\M_i)_{i\in\N}$ such that: 
\begin{enumerate}
\item $\M_i$ is an embedded manifold of dimension $d_i$
\item If $\M_i\cap\partial\M_j\neq \emptyset$, then $\M_i\subset\partial\M_j$ and $d_i<d_j$.
\end{enumerate}

\begin{figure}[h!]
\centering
\includegraphics[scale=1]{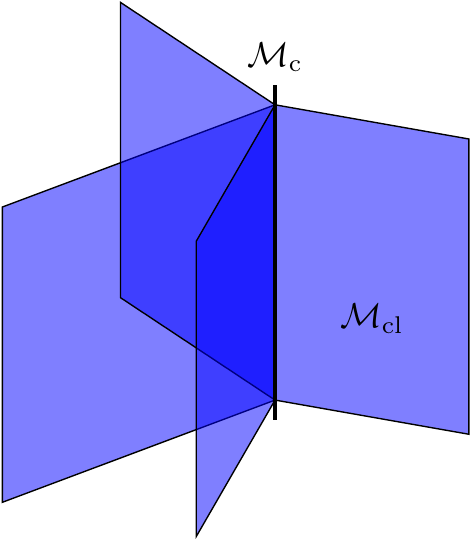}
\caption{Schematic representation of the consensus manifold $\Mc$ (black vertical line) contained in the stratified clustering set $\Scl$ (blue).}\label{fig:consensusmanifold}
\end{figure}

\paragraph{Reminders on consensus achievement.} 
Controlling a group of agents to steer it to consensus has been considered in the literature (see \cite{CFPT13,CFPT15,CPRT17-2,CPRT17, PPS15,PRT15, SCB15}).
One common approach consists of modifying system \eqref{eq:gensyst0} to include an additive control $u\in\UM$:
\begin{equation}\label{eq:gensyst}
\dot x_i = f_i(x) + u_i,\qquad i\in\elts.
\end{equation}
Given $M>0$, we define the set of controls as 
\begin{equation}\label{eq:UM}
\UM:=\big\{ u:\R^+\rightarrow (\R^d)^N \; \big | \; u \text{ measurable, } \sum_{i=1}^N\|u_i(t)\|\leq M \text{ for a.e. } t\in\R^+ \big\}
\end{equation}
where $\|\cdot\|$ is the $\ell_2^d-$Euclidean norm on $\R^d$.
The condition $\sum_{i=1}^N\|u_i(\cdot)\|\leq M$, known as the $\ell_1^N-\ell_2^d$-norm constraint, promotes the componentwise sparsity of the control \cite{CFPT13}.

Previous works (see \cite{CFPT13, CFPT15, PPS15}) have proposed to construct feedback controls by minimizing the variance of the system $V : \R^{dn}\rightarrow\R $, with 
\begin{equation}\label{eq:V}
V(x) = \frac{1}{2N^2}\sum\limits_{i=1}^N\sum\limits_{j=i+1}^N\|x_i-x_j\|^2 \quad \text{ for all } x\in\R^{dn}.
\end{equation}
It is easy to show that the variance characterizes the state of consensus:
\begin{lemma}\label{lemma:variance}
Let $(x_i)_{i\in\{1,\ldots,N\}}\in(\RR^d)^N$, and let $V$ be defined by Equation \eqref{eq:V}.
The system $(x_i(t))_{i\in\{1,\ldots,N\}}$ is in the state of consensus if and only if $V(x(t))=0$.
\end{lemma}
Hence one can choose to design a feedback control strategy by minimizing the time derivative of the variance. Along any trajectory, it satisfies:
\begin{equation*}
\dot{V}(x(t)) = \frac{1}{2N^2}\sum\limits_{i=1}^N\sum\limits_{j\neq i}\langle x_i-x_j, \dot{x}_i-\dot{x}_j\rangle  
= \frac{1}{N}\sum\limits_{i=1}^N\langle (x_i-\bar{x}), f_i(x)+u_i\rangle.
\end{equation*}
As done in \cite{CFPT13} for the Cucker-Smale alignment model, one can easily define a feedback control $u^c$ minimizing the derivative of the variance by setting $i_V:= \argmax\limits_{i\in\{1,\ldots,N\}}\|x_i-\bar{x}\|$ so that:
\begin{equation*}
u_i^c =  \begin{cases}
-M\frac{x_i-\bar{x}}{\|x_i-\bar{x}\|} \; \text{ for } i=i_V \\
0 \; \quad \text{ for all } i\neq i_V.
\end{cases}
\end{equation*}

\paragraph{Searching for a functional to promote declustering.}
On the opposite, the purpose of this work is to design a control strategy leading the system away from clustering. 
Our first approach is to design a control strategy \textit{maximizing} the time derivative of the variance (as opposed to minimizing it when aiming to achieve consensus). 
%Indeed, according to Proposition \ref{prop:variance}, if the variance is being maintained bounded away from $0$, the system avoids the consensus manifold.

\begin{prop}\label{prop:contV}
Let $M>0$ and let $u\in\UM$.
For all $i\in\{1,\ldots,N\}$, let $R_i:=x_i-\bar{x}$, where $\bar{x}:=\frac{1}{N}\sum_{j=1}^N x_j$.
Let $i_V:= \argmax\limits_{i\in\{1,\ldots,N\}}\|R_i\|$.
The control $u^V$ defined by 
\begin{equation}\label{eq:contV}
u^V_i =  \begin{cases}
M\frac{R_{i}}{\|R_{i}\|} \; \text{ for } i=i_V \\
0 \; \text{ for all } i\neq i_V
\end{cases}
\end{equation}
maximizes $\dot{V}$ instantaneously.
\end{prop}
\begin{proof}
We study the time evolution of the variance $V(x(t))=\frac{1}{2N^2}\sum\limits_{i=1}^N\sum\limits_{j=1+1}^N\|x_i(t)-x_j(t)\|^2=\frac{1}{4N^2}\sum\limits_{i=1}^N\sum\limits_{j\neq i} \|x_i(t)-x_j(t)\|^2$. 
Along any trajectory, we have:
\begin{equation*}
\begin{split}
\dot{V} 
&  = \frac{1}{2N^2}\sum\limits_{i=1}^N\sum\limits_{j\neq i}\langle x_i-x_j, f_i(x)-f_j(x)+u_i-u_j\rangle \\
 & = \frac{1}{2N^2}\sum\limits_{i=1}^N\sum\limits_{j\neq i}\langle x_i-x_j, f_i(x)+u_i\rangle -\frac{1}{2N^2}\sum\limits_{i=1}^N\sum\limits_{j\neq i}\langle x_i-x_j, f_j(x)+u_j\rangle \\
 & = \frac{1}{N^2}\sum\limits_{i=1}^N\sum\limits_{j\neq i}\langle x_i-x_j, f_i(x)+u_i\rangle = \frac{1}{N^2}\sum\limits_{i=1}^N\langle\sum\limits_{j=1}^N (x_i-x_j), f_i(x)+u_i\rangle = \frac{1}{N}\sum\limits_{i=1}^N\langle (x_i-\bar{x}), f_i(x)+u_i\rangle.
\end{split}
\end{equation*}
Hence, denoting by $i_V:= \argmax\limits_{i\in\{1,\ldots,N\}}\|x_i-\bar{x}\|$, $\dot{V}$ is maximized at all time by the control given by \eqref{eq:contV}.
%\begin{equation*}
%u^V_i =  \begin{cases}
%M\frac{x_i-\bar{x}}{\|x_i-\bar{x}\|} \; \text{ for } i=i_V \\
%0 \; \quad \text{ for all } i\neq i_V.
%\end{cases}
%\end{equation*}
\end{proof}

%\subsubsection{Controlling the system away from clustering}
%
%We now tackle the more delicate problem of avoiding any clustering of agents. 
Notice that maximizing the variance $V$ will only ensure that the system is far from the consensus manifold, and it does not guarantee declustering. Indeed, $V$ can be very large even if almost all agents are concentrated at one point, as long as one agent is far from the group. 
This calls for the need of a different functional, able to characterize the state of clustering like the variance characterizes consensus.
A natural candidate for that purpose is the entropy functional $W\in C^1(\RR^d\setminus \Scl,\RR)$, defined as follows: 
\begin{equation}\label{eq:W}
W(x) = \frac{1}{N^2}\sum\limits_{i=1}^N\sum\limits_{j=i+1}^N \ln \Vert x_i-x_j\Vert.
\end{equation}
Indeed, if the system is not in the clustering set, the entropy is bounded from below. However, the converse is not true, as we show in the following: 
 \begin{lemma}\label{lemma:W}
 Let $W\in C^1(\RR^d\setminus \Scl,\RR)$ defined by \eqref{eq:W}. If for all $(i,j)\in\elts^2$, $\Vert x_i-x_j\Vert\geq \epsilon$ for some $\epsilon>0$, then $W(x)$ is bounded below, \textit{i.e.} there exists $K(\epsilon)\in\RR$ such that $W(x)>K(\epsilon)$.
 However, the converse does not hold. 
 \end{lemma}
 \begin{proof}
Let $\epsilon>0$ and suppose that for all $(i,j)\in\elts^2$, $\Vert x_i-x_j\Vert\geq \epsilon$. Then 
$$
W(x) \geq \frac{1}{N^2}\sum\limits_{i=1}^N\sum\limits_{j=i+1}^N \ln \epsilon = \frac{N(N-1)}{2N^2}\ln \epsilon .
$$
We now disprove the converse. Let $K>0$. Suppose that 
\begin{equation}\label{eq:Wconverse}
 W(x)\geq K \Rightarrow \exists \epsilon>0 \text{ s.t. } \forall (i,j)\in\elts^2 \text{ s.t. } i<j,  \quad \Vert x_i-x_j\Vert\geq \epsilon,
\end{equation}
where $\epsilon$ does not depend on $x$.
Let $x\in(\RR^d)^N$ such that $W(x)\geq K$.
Let $\beta>1$ and consider $\tilde{x}\in\RR^d$ such that 
\begin{equation*}
\begin{cases}
\Vert \tilde{x}_k- \tilde{x}_l\Vert = \frac{1}{\beta} \Vert x_k-x_l\Vert \\
\Vert \tilde{x}_m- \tilde{x}_n\Vert = \beta \Vert x_m-x_n\Vert \\
\Vert \tilde{x}_i- \tilde{x}_j\Vert = \Vert x_i-x_j\Vert \quad \text{ for all } (i,j)\in\elts^2, \quad i<j, \quad (i,j)\neq (k,l)  \text{ and } (i,j)\neq (m,n).  
\end{cases}
\end{equation*}
Then 
$$
W(\tilde{x}) = W(x) +\frac{1}{N^2} (- \ln( \Vert x_k-x_l\Vert) -  \ln( \Vert x_m-x_n\Vert)  + \ln(\Vert \tilde{x}_k- \tilde{x}_l\Vert) + \ln(\Vert \tilde{x}_m- \tilde{x}_n\Vert) )= W(x). 
$$
Hence $W(\tilde{x})\geq K$ and this result holds independently of $\beta$. Let $\beta = \frac{2\Vert x_k-x_l\Vert}{\epsilon}$. Then $\Vert \tilde{x}_k- \tilde{x}_l\Vert = \frac{\epsilon}{2}$, which contradicts \eqref{eq:Wconverse}.
 \end{proof}
 
Lemma \ref{lemma:W} shows that the functional $W$ cannot characterize the boundedness away from the clustering set. This is due to the fact that the logarithm is unbounded both at zero and at infinity, which allows for the contribution of  small pairwise distances to be compensated by that of large pairwise distances in \eqref{eq:W}.

\paragraph{A good entropy functional for declustering.}
The main idea then is to modify the entropy functional
by replacing the logarithm by a function $g$ bounded at
infinity, in order to characterize clustering.

\begin{definition}\label{def:Wg}
Let $g\in C^1(\RR^{+*}) $ be a strictly increasing function such that $\lim\limits_{s\rightarrow 0}g(s)=-\infty$ and $\lim\limits_{s\rightarrow +\infty}g(s)<\infty$.
We define the generalized entropy functional $W_g$ for system \eqref{eq:Krause} by: 
$$
W_g(t) = \frac{1}{2N^2} \sum\limits_{i=1}^N\sum\limits_{j=i+1}^N g(\|x_i(t)-x_j(t)\|^2).
$$
\end{definition}
The advantage of defining such an entropy functional is that we are able to characterize completely the dispersion of the system. % via the following condition: 

\begin{theorem} \label{th:Wg}
Let $W_g$ be an entropy functional as defined in Definition \ref{def:Wg}. The following two statements are equivalent: 
\begin{enumerate}
\item There exists $\eta>0$ such that for all $t>0$, $W_g(t)>\eta$.
\item There exists $\varepsilon>0$ such that for all $t>0$, for all $(i,j)\in\{1,...,N\}^2$,  $\|x_i(t)-x_j(t)\|>\varepsilon$.
\end{enumerate}
If the conditions above are satisfied, the system is declustered at all time. 
\end{theorem}

\begin{proof}
Let $c\in\RR$. Suppose that for all $(i,j)\in\{1,\ldots,N\}^2$, $\|x_i-x_j\|\geq c.$
Then $W_g = \frac{1}{2N^2}\sum\limits_{i=1}^N\sum\limits_{j> i} g(\|x_i-x_j\|^2) \geq \frac{1}{2N^2}\frac{N(N-1)}{2}g(c^2)$.

Conversely, let $K\in\RR$. Suppose that $W_g\geq K$. Let $m:=\sup\{g(s),\;s>0\}$.
For all $(i,j)\in\{1,\ldots,N\}^2$, $g(\|x_i-x_j\|^2)\leq m$. Let $(k,l)\in\{1,\ldots,N\}^2$ with $k<l$. 
Notice that the assumptions on $g$ given in Definition \ref{def:Wg} imply that $g$ is invertible. Let $g^{-1}:(-\infty,m)\rightarrow (0,+\infty)$ denote the inverse of $g$. 
Since $W_g = \frac{1}{2N^2} \sum\limits_{i=1}^N\sum\limits_{j=i+1}^N g(\|x_i-x_j\|^2)$, we write: 
$$
\|x_k-x_l\|^2 = g^{-1}\left(2N^2 W_g- \sum\limits_{i\neq k } \sum\limits_{j> i, j\neq l } g(\|x_i-x_j\|^2)\right). 
$$
Since $g^{-1}$ is an increasing function, we obtain: 
$$
\|x_k-x_l\|\geq \sqrt{g^{-1}\left(2N^2 K-(\frac{N(N-1)}{2}-1)m\right)}
$$
and the result follows.
\end{proof}

%We aim to design a feedback control strategy to keep the system in a declustered state. 
From Theorem \ref{th:Wg}, maximizing $W_g$ will ensure that the system is declustered, hence that it is far from the clustering set.
%In section \ref{Sec:ContCons} we examined a control strategies maximizing $\dot{V}$ instantaneously.
We design a control strategy to keep the system in a declustered state, by maximizing $\dot{W}_g$ instantaneously.

\begin{prop}\label{prop:contWg}
Let $M>0$ and let $u\in\UM$.
For all $i\in\{1,\ldots,N\}$, let $S_i:=\frac{1}{N}\sum_{j\neq i} g'(\Vert x_i-x_j\Vert^2) (x_i-x_j)$.
Let $i_W:= \argmax\limits_{i\in\{1,\ldots,N\}}\|S_i\|$.
The control $u^W$ defined by 
\begin{equation}\label{eq:contW}
u^W_i =  \begin{cases}
M\frac{S_{i}}{\|S_{i}\|} \; \text{ for } i=i_W \\
0 \; \text{ for all } i\neq i_W
\end{cases}
\end{equation}
maximizes the time derivative of the generalized entropy $\dot W_g$ instantaneously.
\end{prop}
\begin{proof}
Let us start by computing $\dot W_g(t)$.
Since $\dot x_i-\dot x_j = f_i(v)- f_j(v)+u_i- u_j$, we get:
\begin{equation}\label{eq:Wgdot}
\begin{split}
\dot W_g = & \frac{1}{N^2}\sum_{1\leq i< j\leq N} g'(\Vert x_i-x_j\Vert^2) \langle x_i-x_j,f_i(x)- f_j(x)+u_i- u_j\rangle \\
= & \frac{1}{N^2}\sum_{1\leq i< j\leq N} g'(\Vert x_i-x_j\Vert^2) \langle x_i-x_j,f_i(x) +u_i\rangle - \frac{1}{N^2}\sum_{1\leq i< j\leq N} g'(\Vert x_i-x_j\Vert^2) \langle x_i-x_j,f_j(x) +u_j\rangle \\
= & \frac{1}{N^2}\sum_{ i\neq j} g'(\Vert x_i-x_j\Vert^2) \langle x_i-x_j,f_i(x)+u_i\rangle 
=  \frac{1}{N}\sum_{i=1}^N \langle \frac{1}{N}\sum_{j\neq i} g'(\Vert x_i-x_j\Vert^2) (x_i-x_j),f_i(x)+u_i\rangle.
\end{split} 
\end{equation}
Let $S_i:=\frac{1}{N}\sum_{j\neq i} g'(\Vert x_i-x_j\Vert^2) (x_i-x_j)$. Let $i_W:=\arg\max_{i}\|S_i\|$, representing a weighted mean of influences of all agents on agent $i$.
Then the control strategy \eqref{eq:contW} maximizing $\dot{W}_g$ at all time $t$ is sparse. 
\end{proof}

The control strategies designed in Propositions \ref{prop:contV} and \ref{prop:contWg} are both sparse, meaning that the control acts on only one agent at a given time. However, they differ in fundamental ways. In order to maximize the variance $V$, one must act on the agent furthest away from the center of mass of the group, as shown in Proposition \ref{prop:contV}. On the other hand, to maximize the general entropy $W_g$, one must act on an agent which is both close to other agents, and at the edge of the group (as illustrated in Figure \ref{fig:distribRS}).
 
\begin{figure}
\includegraphics[width = 0.45\textwidth]{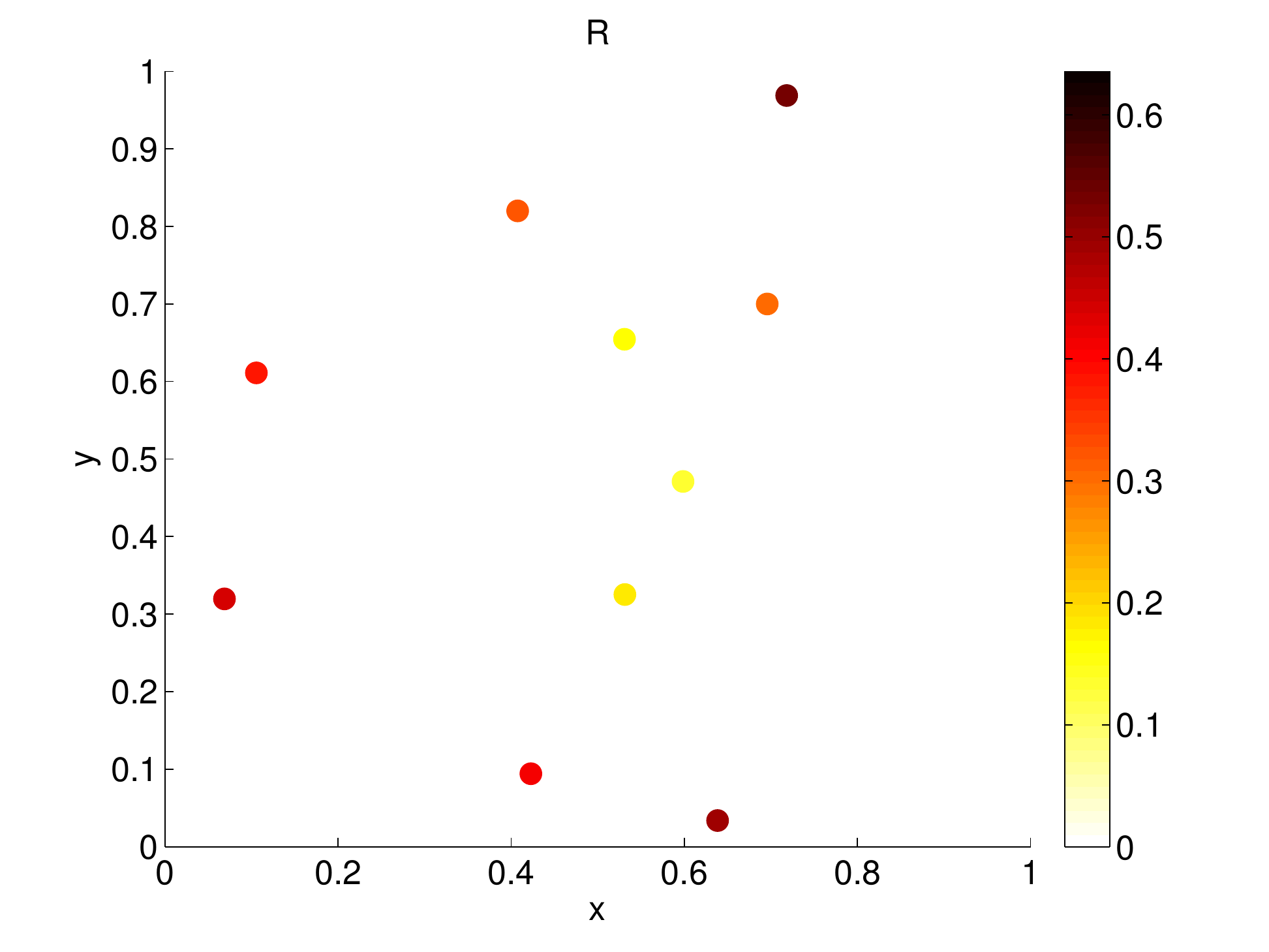}
\includegraphics[width = 0.45\textwidth]{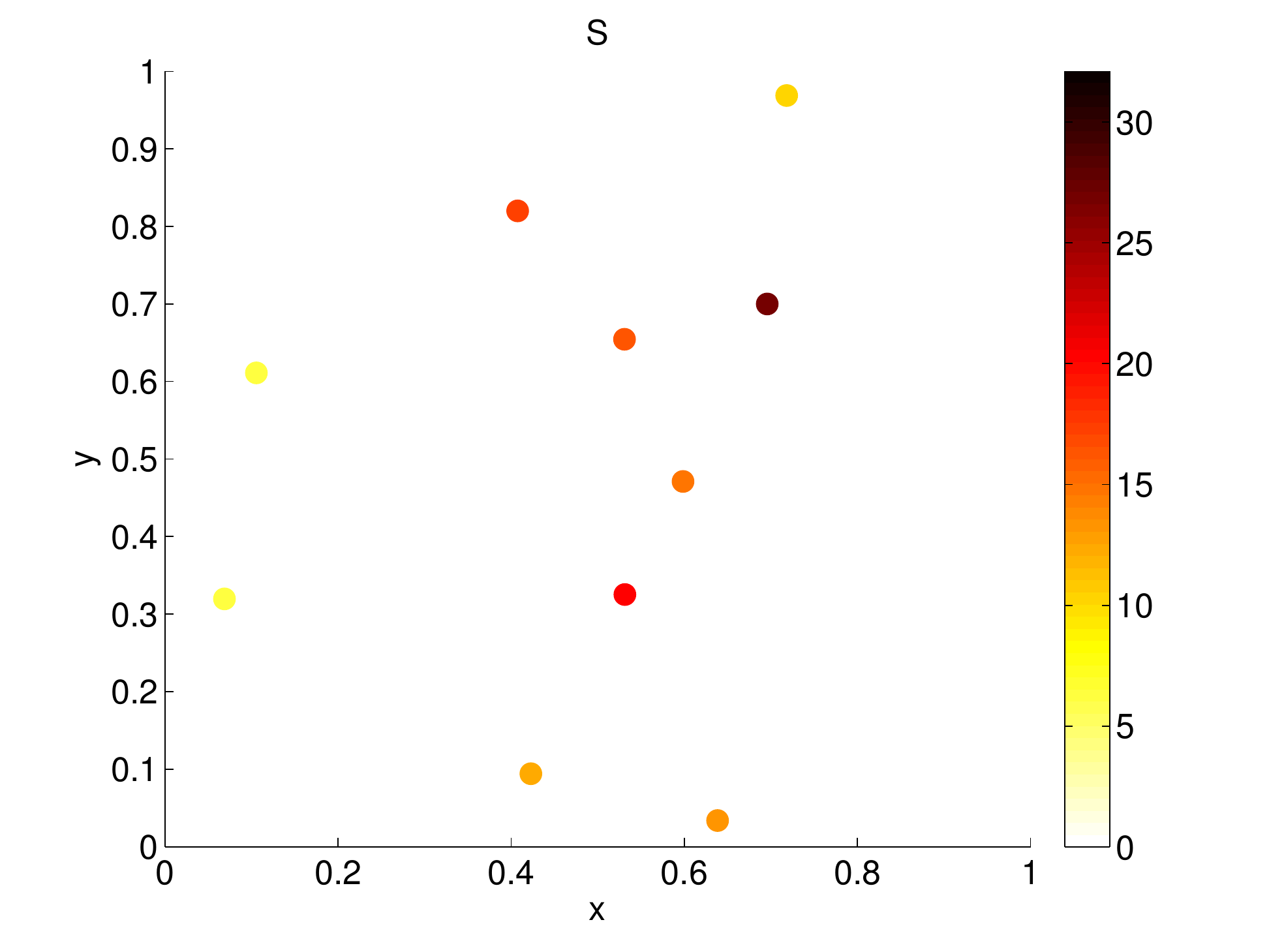}
\caption{Random distribution of $N=10$ agents in $\RR^2$. The control strategies to minimize $\dot{V}$ and $\dot{W}_g$ respectively consist of acting on the agent with the largest $R$ (left) and $S$ (right), computed with $g:s\rightarrow-\frac{1}{s}$.}\label{fig:distribRS}
\end{figure}

In particular these sparse control strategies apply to the Krause system \eqref{eq:Krause1} where $f(x)=\frac{1}{N}\sum_{i\neq j} a(\|x_i-x_j\|)(x_j-x_i)$.

\subsection{Controlling the system away from consensus and clustering}

We now choose to focus our study on the well-known Hegselmann-Krause first order consensus model: 
\begin{equation}\label{eq:Krause1}
\dot{x}_i = \frac{1}{N}\sum_{j\neq i} a(\|x_i-x_j\|)(x_j-x_i), \quad i\in\elts.
\end{equation}
The Hegselmann-Krause model \eqref{eq:Krause1} was designed in the context of opinion dynamics and captures collective behavior such as consensus or clustering \cite{HK02}. In the original ``bounded confidence'' model, each agent aligns its position to an average of all neighbors within a predetermined range.
Here, we generalize this idea by considering that each agent $x_i$ aligns its position to a weighted average of all other agents' positions, depending on the interaction function  
$a:\RR^+\rightarrow \RR^+$. The HK model can be recovered in the special case of $a$ being a step function $s\mapsto a(s) = \one_{s\leq r}$.

Let $M>0$. We define the controlled evolution of the system as follows:
\begin{equation}\label{eq:Krause}
\dot{x}_i = \frac{1}{N}\sum_{j\neq i} a(\|x_i-x_j\|)(x_j-x_i) + u_i, \quad i\in\elts, 
\end{equation}
where $u\in \UM$ (see \eqref{eq:UM}).

Let $\phi_{a,u}:(\RR^d)^n\times\RR\rightarrow (\RR^d)^n$ be the flow associated with the differential equation \eqref{eq:Krause}, \textit{i.e.}, for all $x_0\in (\RR^d)^n$ for all $t\in\RR^+$, $\phi_{a,u}(x_0,t)$ is the unique solution of \eqref{eq:Krause} with initial condition $x(0) = x_0.$

The problem of defining the solution of \eqref{eq:Krause1} when agents collide was treated in \cite{CDFLS11} (Remark 2.10).  
In what follows, we allow the interaction function $a(\cdot)$ to be unbounded near zero. Less restrictive even, we allow the following: $\lim_{s\rightarrow 0} sa(s) = +\infty$. This causes the right-hand side of \eqref{eq:Krause1} to be undefined when two agents cluster. However, we can define the solution up to the time of the first clustering $\bar{t}$. We prove that the limit of the solution of \eqref{eq:Krause1} when approaching $\bar{t}$ is unique. This will allow us to extend the solution in order to give a meaning to the system after a time of clustering.

\begin{lemma}\label{lemma:prolong}
Let $x$ denote the solution of system \eqref{eq:Krause1}, and let $\bar{t}$ be the first time at which a cluster occurs, \textit{i.e.}, for all $t<\bar{t}$, for all $(i,j)\in\elts^2$, $x_i(t)\neq x_j(t)$ 	and there exist $(k,l)\in\elts^2$ such that
$$\lim\limits_{t\rightarrow \bar{t}} \|x_k(t)-x_l(t)\| = 0.$$
There exists $x^L\in(\RR^d)^N$ such that 
$$ \lim\limits_{t\rightarrow \bar{t}} x(t) = x^L.$$
\end{lemma}
\begin{proof}
Let $\bar{t}$ be the first time at which a cluster occurs, and let $x_k$ and $x_l$ be the clustering agents, \textit{i.e.}, for all $t<\bar{t}$, for all $(i,j)\in\elts^2$, $x_i(t)\neq x_j(t)$ 	and 
$\lim_{t\rightarrow \bar{t}} \|x_k(t)-x_l(t)\| = 0.$
Since $a:\RR^+\rightarrow \RR^+$, the system \eqref{eq:Krause1} is contractive and $x$ stays bounded. Hence there exists a sequence $(t_n)$ converging to $\bar{t}$ such that $\lim\limits_{n\rightarrow+\infty} x(t_n) = x^L$.
Suppose that there exist two non-decreasing subsequences $(t_{2n})$ and $(t_{2n+1})$ converging to $\bar{t}$ such that $\lim\limits_{n\rightarrow+\infty} x(t_{2n}) = \tilde x^L$ and $\lim\limits_{n\rightarrow+\infty} x(t_{2n+1}) = \hat x^L$.  
Suppose that for $i\in\elts \setminus \{k,l\}$, $\tilde x^L_i\neq \hat x^L_i$. Then there exists $\epsilon>0$ such that for all $n\in\N$ big enough
\begin{equation}\label{eq:prooflimit}
 \| x_i(t_{2n})-x_i(t_{2n+1}) \|\geq \epsilon.
\end{equation}
Since $x_i$ is not part of a cluster, there exists $\eta>0$ such that for all $j\in\elts$, $\|x_i-x_j\|\geq c$, which implies that there exists $A>0$ such that for all $j\in\elts$, $\|x_i-x_j\|\leq A$. Since $x$ is bounded, 
there also exists $C>0$ such that for all $j\in\elts$, $\|x_i-x_j\|\leq C$. Hence from \eqref{eq:Krause1}, we get:
$\|\dot x_i \| \leq \frac{1}{N} \sum_{j\neq i} AC \leq AC.$
Then $$ \| x_i(t_{2n})-x_i(t_{2n+1}) \| \leq AC | t_{2n}-t_{2n+1} | \underset{n\rightarrow+\infty}\longrightarrow 0 $$ which contradicts \eqref{eq:prooflimit}. Therefore, for all $i\in\elts \setminus \{k,l\}$, $\tilde x^L_i = \hat x^L_i$.

Now suppose that $\tilde x_k^L\neq \hat x_k^L$. Since $\tilde x_k^L = \tilde x_l^L$ and $\hat x_k^L = \hat x_l^L$, it automatically holds: $\tilde x_l^L\neq \hat x_l^L$.
Now notice that one characteristic of System \eqref{eq:Krause1} is that the mean $\bar{x}$ stays constant in time.
We should then have $\bar{\tilde x}^L=\bar{\hat  x}^L$. However, we have 
$$
\bar{\tilde x}^L = \frac{1}{N} \sum_{i=1}^N \tilde x_i^L = \frac{1}{N} \sum_{i\neq k,l} \tilde x_i^L + \frac{2}{N}\tilde x_k^L = \frac{1}{N} \sum_{i\neq k,l} \bar x_i^L + \frac{2}{N} \tilde x_k^L \neq \bar{\hat x}^L.
$$
This proves that there exists a unique limit $x^L = \tilde{x}^L = \bar{x}^L = \lim\limits_{t\rightarrow \bar{t}} x(t) $.
\end{proof}

Since \eqref{eq:Krause1} may be undefined when two or more agents cluster, we impose $a(0) = 0$.
% and rewrite it as follows: 
%\begin{equation}\label{eq:Krause2}
%\dot{x}_i = \frac{1}{N}\sum_{j \in c(i)} a(\|x_i-x_j\|)(x_j-x_i), \quad i\in\elts.
%\end{equation}
%where $C(i)=\{j\in\elts, \; j\neq i, \; x_j(t)\neq x_i(t)\}$, as done in \cite{CDFLS11}. 
Lemma \ref{lemma:prolong} implies that the solution can be extended after each time of clustering. The condition $a(0) = 0$ implies that once two agents collide, they stay clustered (for the system without control). 

\subsubsection{Black Hole}\label{Sec:BlackHole}

In Section \ref{Sec:Entropy}, we designed a control strategy in the general case of system \eqref{eq:gensyst}. We now study the more specific first-order consensus model \eqref{eq:Krause}.
In this section, we prove that for certain interaction functions $a(\cdot)$, there exists a\textit{ black hole}, \textit{i.e.}, given a certain bound $M$ on the control (with $\sum_{i=1}^N\|u_i\|\leq M$), for certain initial conditions, it is impossible to avoid convergence to consensus whatever the control may be. This is a manifestation of the Black Swan phenomenon. 

\begin{definition}
Let $M>0$. 
We define the \textit{black hole region} as follows:
$$
\mathcal{R}^M_\mathrm{BH} = \{x_0\in(\RR^d)^n\; | \; \forall u\in\UM, \; \exists T>0, \; V(\phi_{a,u}(x_0,T))=0 \}.
$$
\end{definition}

\begin{theorem}\label{th:BH}
Let $a$ be an attraction potential such that 
$\lim_{s\rightarrow 0} sa(s) = + \infty  $. Then for all $M>0$, there exists $\epsilon>0$ such that if for all $(i,j)\in\elts^2$, $\|x_i(0)-x_j(0)\|<\epsilon$, then given any control $u\in\UM$, the system converges to consensus in finite time.  In other words, for any $M>0$, there exists $\mathcal{R}^M_\mathrm{BH}$ such that $\mathcal{M}_\mathrm{c} \varsubsetneq \mathcal{R}^M_\mathrm{BH}$.
\end{theorem}
\begin{proof}
 We study the evolution of the variance $V(t)=\frac{1}{2N^2}\sum\limits_{i=1}^N\sum\limits_{j=1+1}^N\|x_i(t)-x_j(t)\|^2=\frac{1}{4N^2}\sum\limits_{i=1}^N\sum\limits_{j\neq i} \|x_i(t)-x_j(t)\|^2$. 
 Along any trajectory of $x$, $\dot{V} $ satisfies:
\begin{equation*}
\begin{split}
\dot{V} 
& = \frac{1}{2N^2}\sum\limits_{i=1}^N\sum\limits_{j\neq i}\langle x_i-x_j, \frac{1}{N} \sum_{k\neq i} a(\|x_i-x_k\|)(x_k-x_i)- \frac{1}{N} \sum_{k\neq j} a(\|x_j-x_k\|)(x_k-x_j)+u_i-u_j\rangle.
\end{split}
\end{equation*}
The uncontrolled part of $\dot{V}$ writes:
\begin{equation*}
\begin{split}
 &   \frac{1}{2N^2}\sum\limits_{i=1}^N\sum\limits_{j\neq i}\langle x_i-x_j, \frac{1}{N} \sum_{k\neq i} a(\|x_i-x_k\|)(x_k-x_i) \rangle 
-  \frac{1}{2N^2}\sum\limits_{i=1}^N\sum\limits_{j\neq i} \langle x_i-x_j, \frac{1}{N} \sum_{k\neq j} a(\|x_j-x_k\|)(x_k-x_j) \rangle  \\
= & \frac{1}{2N^3}  \left( \sum\limits_{i=1}^N \sum\limits_{j\neq i}\sum_{k\neq i} \langle x_i-x_k,a(\|x_i-x_j\|)(x_j-x_i)\rangle
- \sum\limits_{j=1}^N\sum\limits_{i\neq j}\sum_{k\neq j} \langle x_k-x_j,a(\|x_i-x_j\|)(x_i-x_j)\rangle \right)\\
= & \frac{1}{2N^3} \sum\limits_{i=1}^N\sum\limits_{j\neq i} \bigg( \sum\limits_{k\neq i,j}\langle x_i-x_j,a(\|x_i-x_j\|)(x_j-x_i)\rangle \\
& \hspace{3cm} +
\langle x_i-x_j,a(\|x_i-x_j\|)(x_j-x_i)\rangle - \langle x_i-x_j,a(\|x_i-x_j\|)(x_i-x_j)\rangle \bigg)\\
= & \frac{1}{2N^2} \sum\limits_{i=1}^N\sum\limits_{j\neq i} \langle x_i-x_j,a(\|x_i-x_j\|)(x_j-x_i)\rangle 
= - \frac{1}{2N^2} \sum\limits_{i=1}^N\sum\limits_{j\neq i} a(\|x_i-x_j\|)\|x_i-x_j\|^2 \\
= & - \frac{1}{N^2} \sum\limits_{i=1}^N\sum\limits_{j= i+1}^N a(\|x_i-x_j\|)\|x_i-x_j\|^2.
\end{split}
\end{equation*}
Let $M>0$. Since $\lim_{s\rightarrow 0} sa(s) = + \infty  $, for all $A>0$, there exists $\epsilon>0$ such that for all $s<\epsilon, \; a(s)\geq \frac{A}{s}$. 
Near consensus, that is when for all $i$ and $j$, $\|x_i(t)-x_j(t)\|\leq\epsilon$ :
\begin{equation}\label{eq:VdotBH}
\begin{split}
\dot{V} 
& = -\frac{1}{N^2}\sum\limits_{i=1}^N\sum\limits_{j= i+1}^N a(\|x_i-x_j\|)\|x_i-x_j\|^2+\frac{1}{2N^2} \sum\limits_{i=1}^N\sum\limits_{j= i+1}^N \langle x_i-x_j,u_i-u_j\rangle \\
& \leq -\frac{1}{N^2}A \sum\limits_{i=1}^N\sum\limits_{j= i+1}^N \|x_i-x_j\|+\frac{1}{2N^2}2M\sum\limits_{i=1}^N\sum\limits_{j= i+1}^N\| x_i-x_j\|.
\end{split}
\end{equation}
In particular, let $A=2M$ and let $\epsilon>0$ such that for all $s<\epsilon$, $a(s)\geq \frac{A}{s}$.
Notice that $V\geq \frac{1}{2N^2}\max\limits_{i,j}\|x_i-x_j\|^2$.
Suppose that $V(0) = \frac{\epsilon^2}{2N^2}$. Then for all $(i,j)\in\{1,\ldots N\}^2$, $\|x_i(0)-x_j(0)\|\leq \sqrt{2N^2V(0)} = \epsilon$.
Then while $\|x_i-x_j\|\leq \epsilon$,
$$
\dot{V}\leq -\frac{M}{N^2}\sum\limits_{i=1}^N\sum\limits_{j= i+1}^N\| x_i-x_j\|.
$$
Recall that by equivalence of the norms, 
\begin{equation}\label{eq:normequivalence}
 \sqrt{\sum\limits_{i=1}^N\sum\limits_{j= i+1}^N\| x_i-x_j\|^2}
\leq \sum\limits_{i=1}^N\sum\limits_{j= i+1}^N \|x_i-x_j\|
\leq \sqrt{\frac{n(n-1)}{2}} \sqrt{\sum\limits_{i=1}^N\sum\limits_{j= i+1}^N\| x_i-x_j\|^2}.
\end{equation}
So 
$$
\dot{V}\leq -\frac{M}{N^2} \sqrt{\sum\limits_{i=1}^N\sum\limits_{j= i+1}^N\| x_i-x_j\|^2}  
= -\frac{M}{N^2} \sqrt{2N^2 V}
= -\frac{\sqrt{2}M}{N}\sqrt{V}.
$$
Then $V$ decreases which ensures that the condition $\|x_i-x_j\|\leq\epsilon$ holds.
Hence $V$ tends to $0$ in finite time.
\end{proof}

\begin{rem}
The condition $\lim_{s\rightarrow 0} sa(s) = + \infty$ does not generalize to the integral condition: $\int_0^{s_0}a(s)ds=+\infty$ given in \cite{CFPT13, HHK10}.
Take for instance $a(s)=\frac{1}{s}$. Then $\int_0^{s_0}a(s)ds=+\infty$, but $\lim_{s\rightarrow 0} sa(s) =1$. Indeed, going back to the proof above, the derivative of the variance satisfies: 
$$
\dot{V} = -\frac{1}{N^2} \sum\limits_{i,j} \| x_i-x_j\| + \frac{1}{2N^2} \sum\limits_{i,j} \langle x_i-x_j,u_i-u_j\rangle \leq \frac{-1+M}{N^2} \sum\limits_{i,j} \| x_i-x_j\| .
$$
If $M<1$, then convergence to consensus is unavoidable, but for bigger values of $M$ the possibility of acting on the system to prevent consensus remains.
\end{rem}
We now generalize Theorem \ref{th:BH} for functions $s\mapsto sa(s)$ that are bounded below for small values of $s$. We prove that the existence of a black hole depends on the value of the bound $M$ on the control, unlike in the case of Theorem \ref{th:BH} where a black hole exists no matter how strong the control is allowed to be. 

\begin{theorem}
Let $a$ be an attraction potential such that for $s\leq \epsilon$, $sa(s)\geq C$. Then if $M<C$, there exists a black hole.
\end{theorem}
\begin{proof}
Suppose that $V(0)\leq \frac{\epsilon}{\sqrt{2}N}$. Then for all $(i,j)\in\{1,\ldots N\}^2$, $\|x_i(0)-x_j(0)\|\leq\sqrt{2}NV(0) = \epsilon$.
Then from \eqref{eq:VdotBH} and \eqref{eq:normequivalence}, while $\|x_i-x_j\|\leq\epsilon$, we have:
$$
\dot{V}\leq \frac{M-C}{N^2}\sum\limits_{i=1}^N\sum\limits_{j= i+1}^N\| x_i-x_j\|\leq \frac{\sqrt{2}(M-C)}{N}\sqrt{V} .
$$
Since $M-C<0$, this ensures that $V$ decreases and that the condition $\|x_i-x_j\|\leq \epsilon$ holds. Hence $V$ tends to 0 in finite time.
\end{proof}

Theorem \ref{th:BH} shows that if the interaction between agents is very strong when they are close to each other (as characterized by the condition $\lim_{s\rightarrow 0} sa(s) = + \infty$), then for every bound $M$ on the control, there exists a zone close to the consensus manifold such that no control in $\UM$ can prevent consensus. We call this phenomenon the \textit{black hole}.
We now look at the behavior of the system far from the clustering set, that is when each pair of agents is sufficiently separated. We show in Sections \ref{Sec:Safety} and \ref{Sec:basin} that depending on the strength of the decrease of $a$ near infinity, there may or may not exist a \textit{safety region} far from the consensus manifold, that is a stable zone (given appropriate control). 

\subsubsection{Safety Region}\label{Sec:Safety}

Here we give sufficient conditions on the potential for the existence of a \textit{safety region}. Given a bound $M$ on the control, there exist initial conditions such that the control can always keep the system away far from clustering.

\begin{definition}
Let $M>0$. 
Construct $W_g$ as in Definition \ref{def:Wg}.
We define the \textit{safety region} as follows:
$$
\mathcal{R}^M_\mathrm{S} = \{x_0\in(\RR^d)^n\; | \; \exists u\in\UM, \; \exists K\in\RR, \; \forall t\geq 0, \; W_g(\phi_{a,u}(x_0,t))\geq K \}.
$$
\end{definition}

\begin{rem}
Notice that from Theorem \ref{th:Wg}, the safety region is equivalently defined by:
$$
\mathcal{R}^M_\mathrm{S} = \{x_0\in(\RR^d)^n\; | \; \exists u\in\UM, \; \exists \epsilon>0, \; \forall t\geq 0, \; \forall i\neq j, \; \|\phi_{a,u}(x_0,t)_i-\phi_{a,u}(x_0,t)_j\|\geq \epsilon\}.
$$
\end{rem}

\begin{theorem}\label{th:safety1}
Let $a$ be an attraction potential such that 
$\lim_{s\rightarrow +\infty} sa(s) = 0 $. 
Then for all bound $M>0$ on the control, there exists a safety region $\mathcal{R}^M_\mathrm{S}\neq \emptyset$. Furthermore, confinement to the safety region can be obtained with the sparse control $u^W\in\UM$ given in \eqref{eq:contW}.
\end{theorem}

%\begin{rem}
%According to Theorem \ref{th:Wg}, the condition $\dot{W}_g>0$ is enough to ensure that the system remains far from the consensus manifold at all time. 
%\end{rem} 

\begin{proof}
From \eqref{eq:Wgdot}, $\dot{W}_g$ is maximized instantaneously by the sparse control \eqref{eq:contW}, and we have:
\begin{equation}\label{eq:Wgdot2}
\max_{u} \dot{W}_g = \frac{1}{N}\sum_{i=1}^N\langle S_i, \frac{1}{N} \sum_{k=1}^N a(\|x_i-x_k\|)(x_i-x_k) \rangle + \frac{M}{N}\|S_{i_W}\|.
\end{equation}

%Furthermore, due to the monotony of the chosen function $g$, for all $\bar W$, there exists $\mu_2(\bar W)>0$ such that if $W_g\geq\bar W$, then for all $i,j$, $\|x_i-x_j\|\geq\mu_2(\bar W)$.

Let $\epsilon<\frac{M}{N}$. Since $\lim_{s\rightarrow +\infty} sa(s) = 0 $, there exists $\mu>0$ such that if for all $i, j$, $\|x_i-x_j\|\geq \mu$, then $\frac{1}{N} \sum_{k=1}^N a(\|x_i-x_k\|)\|x_i-x_k\|\leq\epsilon$. 
 Suppose that at $t=0$, the initial conditions give: $W_g(0)\geq \frac{1}{2N^2} g(\mu^2) + \frac{m}{2N^2}(\frac{N(N-1)}{2}-1)$.
%the agents are spread out enough that for all $i,j$, $\|x_i(0)-x_j(0)\|\geq \mu_2(W_g(0))\geq\mu_1(\epsilon)$. 
Then from the proof of Theorem \ref{th:Wg}, for all $(i,j)\in \{1,\ldots,N\}^2$, 
$$
\|x_i(0)-x_j(0)\|\geq \sqrt{g^{-1}\left( g(\mu^2)+m(\frac{N(N-1)}{2}-1)-m(\frac{N(N-1)}{2}-1)\right)} = \mu.
$$
Then $\max_{u} \dot{W}_g \geq \|S_{i_W}\|(\frac{M}{N}-\epsilon)\geq 0$. 
Then, choosing the control $u^W$ given by \eqref{eq:contW} that maximizes $\dot{W}_g$ at all time, we ensure that for all $t>0$, $W_g(t)\geq W_g(0)$, so that for all $i,j$, $\|x_i(t)-x_j(t)\|\geq\mu$. Furthermore, $u^W$ is sparse.
\end{proof}

This first theorem covers a wide range of interaction potentials. Given that the interaction potential $a(\cdot)$ decreases enough at infinity, we ensure the existence of a safety zone far from the clustering set. This for instance applies to potentials $a(\cdot)$ with compact support. However, notice that the interaction potential $a(s)=\frac1s$ does not meet the required conditions of Theorem \ref{th:safety1}. Here we state a new theorem dealing with functions that decrease at the speed of $1/s$. 

\begin{theorem}\label{th:safety2}
Let $C\in\RR$ and let $a$ be an attraction potential such that for $s$ large enough, $sa(s) \leq C$. Let $W_g$ be the generalized entropy constructed as in Definition \ref{def:Wg}. If $M>CN$, then there exists a safety zone in which $W_g$ is increasing. Furthermore, confinement to the safety zone can be obtained with the sparse control $u^W\in\UM$ defined by \eqref{eq:contW}.
\end{theorem}
\begin{proof}
The proof follows the same structure as the proof of Theorem \ref{th:safety1}. 
Let $0<\epsilon<M-CN$. Since for $s$ large enough, $ sa(s) \leq C $, there exists $\mu>0$ such that if for all $i, j$, $\|x_i-x_j\|\geq \mu$, then $\frac{1}{N} \sum_{k=1}^N a(\|x_i-x_k\|)\|x_i-x_k\|\leq C+\epsilon$. 
 Suppose that at $t=0$, the initial conditions give: $W_g(0)\geq \frac{1}{2N^2} g(\mu^2) + \frac{m}{2N^2}(\frac{N(N-1)}{2}-1).$
%the agents are spread out enough that for all $i,j$, $\|x_i(0)-x_j(0)\|\geq \mu_2(W_g(0))\geq\mu_1(\epsilon)$. 
Then by Theorem \ref{th:Wg}, for all $(i,j)\in \{1,\ldots,N\}^2$, 
$$
\|x_i(0)-x_j(0)\|\geq \sqrt{g^{-1}\left( g(\mu^2)+m(\frac{N(N-1)}{2}-1)-m(\frac{N(N-1)}{2}-1)\right)} = \mu.
$$
Then using the sparse control $u^W$ defined in \eqref{eq:contW}, we maximize $\dot W_g$ instantaneously as in \eqref{eq:Wgdot2}, with $\max_{u} \dot{W}_g(0) \geq \|S_{i_W}\|(\frac{M}{N}-(C+\epsilon))$. 
If $M>CN$, we can choose a control strategy achieving $\dot{W}_g\geq 0$ at all time, ensuring that for all $t>0$, $W_g(t)\geq W_g(0)$, so that for all $i,j$, $\|x_i(t)-x_j(t)\|\geq\mu$.
\end{proof}

\begin{rem}
The improvement of Theorem \ref{th:safety2} over Theorem \ref{th:safety1} lies in the limit cases of the type $a:s\mapsto\frac{1}{s}$.
In that case, $C=1$ and if $M>N$, then there exists a safety zone far from consensus. 
\end{rem}

\paragraph{Black hole horizon.}
In Sections \ref{Sec:BlackHole} and \ref{Sec:Safety}, we showed the existence of a \textit{black hole} in a neighborhood of the consensus manifold if $\lim_{s\rightarrow 0} sa(s) = + \infty  $ and the existence of a \textit{safety region} far from the clustering set  if $\lim_{s\rightarrow +\infty} sa(s) = 0 $. This suggests the existence of a ``horizon'' between safety and attraction to the black hole for interaction potentials that meet both conditions. The question remains of clarifying this horizon.

\begin{definition}
We define the \emph{black hole horizon} $ \mathcal{H}^M_{\mathrm{BH}}$ as the subset of $(\RR^{d})^N$ given by:
$$\mathcal{H}^M_{\mathrm{BH}} := (\RR^d)^N \setminus (\mathcal{R}^M_\mathrm{BH}\cup \mathcal{R}^M_\mathrm{S}).$$ 
If there is no \emph{safety region} and $\mathcal{R}^M_\mathrm{BH}=(\RR^d)^N$, we say that the \emph{black hole horizon} is \emph{infinite}.\\
If $\mathcal{H}^M_{\mathrm{BH}}=\emptyset$ while $\mathcal{R}^M_\mathrm{BH}\neq \emptyset$ and $\mathcal{R}^M_\mathrm{S}\neq \emptyset$, then the state space $(\RR^d)^N$ is divided between the black hole and the safety region, and we say that the \emph{black hole horizon} is \emph{sharp}.
\end{definition}
The schematic of Figure \ref{fig:horizon} illustrates the black hole horizon enclosed between the safety region and the black hole. 
\begin{figure}
\centering
\begin{tikzpicture}
\draw[fill, color={rgb:black,1;white,2}, draw=black] (4.5,2) ellipse[x radius = 3, y radius = 1.5];
\node [at={(1,1.5)}] {$\mathcal{R}^M_\mathrm{S}$};
%\draw[fill, color=grey] 
\draw [draw, fill, color=white, draw=black] (5,2) ellipse [x radius =2, 
y radius = 1];
\node [at = {(4,2.5)}] {$\mathcal{R}^M_\mathrm{BH}$};
\node [at = {(2,2)}] {$\mathcal{H}^M_{\mathrm{BH}}$};
\end{tikzpicture}
\caption{Schematic representation of the blach hole $\mathcal{R}^M_\mathrm{BH}$, the safety region $\mathcal{R}^M_\mathrm{S}$ and the black hole horizon $\mathcal{H}^M_{\mathrm{BH}}$.}\label{fig:horizon}
\end{figure}
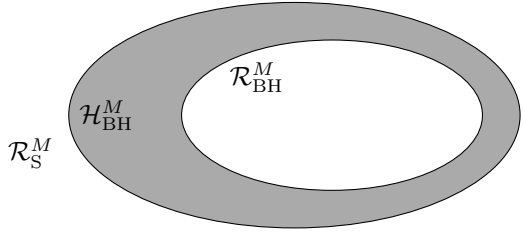

If the attraction potential does not satisfy the hypotheses of Theorems \ref{th:safety1} or \ref{th:safety2}, we cannot ensure the existence of a \textit{safety region}. In fact, we show that in certain cases the \textit{safety region} does not exist and the whole space is a \textit{black hole}, \textit{i.e.}, the \textit{black hole horizon} is infinite.

\begin{lemma}
If $a(s)=1+\frac{1}{s^2}$, there exists $M>0$ such that the \emph{black hole horizon} is infinite. 
\end{lemma}

\begin{proof}
Let $M\leq \frac{\alpha}{\sqrt{2}}$ for some $\alpha<1$. \\
First assume that initially $\sqrt{V(0)}\leq\sqrt{2}M$ (\textit{i.e.}, some agents are already close to each other). We study the evolution of the variance $V=\frac{1}{2N^2}\sum_{i=1}^N\sum_{j=1}^N\|x_i-x_j\|^2$: 
\begin{equation}
\frac{dV}{dt} = -\frac{1}{N^2}\sum_{i=1}^N\sum_{j=1}^N a(\|x_i-x_j\|)\|x_i-x_j\|^2 + \frac{1}{N^2}\sum_{i=1}^N\sum_{j=1}^N \langle x_i-x_j, u_i-u_j\rangle .
\end{equation}
The second term is related to $V$ by equivalence of the norms: 
$$
 \frac{1}{N^2}\sum_{i=1}^N\sum_{j=1}^N \langle x_i-x_j, u_i-u_j\rangle \leq M \frac{1}{N^2}\sum_{i=1}^N\sum_{j=1}^N \| x_i-x_j \| \leq M \frac{1}{N^2} N \sqrt{ \sum_{i=1}^N\sum_{j=1}^N \| x_i-x_j \|^2} = M\sqrt{2V}.
$$
Since $a(s)\geq \frac{1}{s^2}$, while $\sqrt{V}\leq\sqrt{2}M$ we have
\begin{equation}
\frac{dV}{dt} \leq -1 + M\sqrt{2V}\leq -1+2M^2 \leq \alpha -1 , 
\end{equation}
so $V$ converges to 0 in finite time. \\
Let us now suppose that $\sqrt{V(0)} > \sqrt{2}M$, so the initial conformation is far from the consensus manifold. While this condition is satisfied, since $a(s)\geq 1$, we write: 
\begin{equation}
\frac{dV}{dt} \leq -2V+M\sqrt{2V} = \sqrt{V}(-2\sqrt{V}+\sqrt{2}M) \leq -\sqrt{V} (\sqrt{2}M).
\end{equation}
So $V$ decreases until $\sqrt{V} = \sqrt{2}M$. When that happens, we are brought back to the first case.
\end{proof}

\subsubsection{Basin of attraction}\label{Sec:basin}

In Theorems \ref{th:safety1} and \ref{th:safety2}, we saw that if $a(\cdot)$ decreases fast enough to $0$ at infinity, then there exists a ``safety'' zone near infinity (\textit{i.e.}, when the agents are far from each other). Here we show that this safety zone does not always exist.

\begin{theorem}\label{th:basin}
Let $M>0$. If $\lim\limits_{s\rightarrow +\infty} sa(s)=+\infty$, then there exists a real number $\mu>0$ and a set $$B:=\{(x_i)_{i\in\{1,\ldots,N\} } \; | \min\limits_{i\neq j} \|x_i-x_j\|\leq \mu \}$$ such that for any initial condition $x(0)$, there exists a finite time $T>0$ such that $x(T)\in B$, for any control $u\in\UM$.
%In other words, there exist two sets $B_1$ and $B_2$ with $B_1\subset B_2$, such that for all $x(0)\in \RR^{dN}\setminus B_1$, for all control $u\in \UM$ , there exists $T>0$ such that for all $t\geq T$, $v_u(t)\in B_2$.  
We call $B$ the \emph{basin of attraction}. 
\end{theorem}

\begin{proof}
Let $A>M$. There exists $\mu>0$ such that if $s>\mu$, then $sa(s)>A$.
Let $$B:=\{(x_i)_{i\in\{1,...,N\}}\in\RR^{dN} \; | \min\limits_{i\neq j} \|x_i-x_j\|\leq \mu\}$$
and
 $$B^c := \RR^{dN} \setminus B = \{(x_i)_{i\in\{1,...,N\}}\in\RR^{dN} \; | \forall (i,j)\in\{1,...,N\}^2, \; i\neq j \implies \|x_i-x_j\|> \mu\}.$$  
Suppose that for all time $t>0$, $x(t)\in B^c$.
Then the variance $V$ decreases as a quadratic function of time. Indeed,
$$
\dot{V} = -\frac{1}{N^2}\sum_{i<j} a(\|x_i-x_j\|)\|x_i-x_j\|^2 + \frac{1}{N^2}\sum_{i<j}\langle x_i-x_j, u_i-u_j\rangle \leq \frac{M-A}{N^2}\sum_{i<j} \|x_i-x_j\| \leq \frac{\sqrt{2}(M-A)}{N}\sqrt{V}.
$$
%Since $A>M$, by equivalence of the norms, there exists $\gamma>0$ such that $\dot{V}\leq -\gamma\sqrt{V}$.
Then there exists $\tau>0$ such that $V(\tau)\leq \frac{N-1}{4N} \mu^2$.
This contradicts the statement: $x(\tau)\in B^c$, as it would imply: $V(\tau) = \frac{1}{2N^2}\sum_{i=1}^N\sum_{j= i+1}^N \|x_i-x_j\|^2 > \frac{N-1}{4N}\mu^2$.  
Hence there exists $T>0$ such that $x(T)\in B$.

%Hence for all $t>T$, $V(t)\leq\frac{s_0}{2}$. 
%This implies that for all $t>T$, $ \frac{1}{2N^2} \max_{i,j}\|x_i-x_j\| \leq \frac{s_0}{2}$. So for all $t>T$, $v(t)\in B_2$, where 
% $B_2:=\{(x_i)_{i\in\{1,...,N\}}\in\RR^{dN} \; | \forall i,j\in\{1,...,N\}, \|x_i-x_j\|\leq N^2 s_0\}$. 
\end{proof}

\begin{rem}
Theorem \ref{th:basin} does not ensure that $x(t)$ stays in the basin of attraction $B$ for all time $t>T$. 
When $x\in B$, it might become possible to act on the system again. If the control allows to obtain again $x\in B^c$, again $V$ becomes strictly decreasing until $x\in B$. 
\end{rem}

As in Sections \ref{Sec:BlackHole} and \ref{Sec:Safety}, we extend Theorem \ref{th:basin} to the case of functions that only satisfy: $sa(s)\geq C$ for $s$ large enough. 

\begin{theorem}\label{th:basin2}
Let $M>0$, and let $C>M$. If there exists $\mu>0$ such that $sa(s)\geq C$ for all $s>\mu$, 
%, and if $M<C$, then there exists a real number $\mu>0$ and a set  such that 
there exists a basin of attraction 
$$B:=\{(x_i)_{i\in\{1,\ldots,N\} } \; | \min\limits_{i\neq j} \|x_i-x_j\|\leq \mu \},$$ 
\textit{i.e.}
 for any initial condition $x(0)$, there exists a finite time $T>0$ such that $x(T)\in B$,
for any control $u\in\UM$.
\end{theorem}

\begin{proof}
Suppose that there exists $\mu\in\RR$ such that for all $s>\mu$, $sa(s)\geq C$. Suppose that the bound on the control is such that $M<C$.
As in the proof of theorem \ref{th:basin}, let $B:=\{(x_i)_{i\in\{1,...,N\}}\in\RR^{dN} \; | \min\limits_{i\neq j} \|x_i-x_j\|\leq \mu\}$ and $B^c:=\RR^{dN}\setminus B$.
Since $M<C$, the variance again decreases as a quadratic function of time. 
$$
\dot{V} = -\frac{1}{N^2}\sum_{i<j} a(\|x_i-x_j\|)\|x_i-x_j\|^2 + \frac{1}{N^2}\sum_{i<j}\langle x_i-x_j, u_i-u_j\rangle \leq \frac{M-C}{N^2}\sum_{i<j} \|x_i-x_j\| \leq \frac{\sqrt{2}(M-C)}{N}\sqrt{V}.
$$
As in the proof of Theorem \ref{th:basin}, this implies that there exists $T>0$ such that $x(T)\in B$.
\end{proof}

\begin{rem}
Notice that Theorems \ref{th:safety2} and \ref{th:basin2} cannot be used together. Indeed, if for $s$ large enough, $C_1 \leq sa(s)\leq C_2$, Theorem \ref{th:safety2} applies if $M>NC_2$ and Theorem \ref{th:basin2} applies if $M<C_1$, which is never satisfied.  
\end{rem}

%\begin{rem}
%In the case $d=1$, $N=2$, the consensus manifold is the line $x_1=x_2$. The sets $B_1$ and $B_2$ are defined as: $B_1=\{(v_1,v_2)\in\RR^2 \; | \; \|v_1-v_2\|\leq s_0\}$ and
%$B_2=\{(v_1,v_2)\in\RR^2 \; | \; \|v_1-v_2\|\leq 4 s_0\}$.
%\end{rem}
%
%\begin{figure}[h]
%\begin{center}
%\begin{tikzpicture}[scale=0.6]
%\definecolor{gray1}{gray}{0.8}
%\definecolor{gray2}{gray}{0.9}
%\draw[scale = 1,domain=-3.5:4,smooth,variable=\x,blue] plot ({\x},{\x-0.5});
%\fill [ gray2]
%      (-2, -4)
%      -- (4,2)
%	  -- (4,4)
%	  -- (2,4)
%	  -- (-4, -2)
%	  -- (-4, -4)
%	  -- (-2,-4)
%      -- cycle;
%\fill [ gray1]
%      (-3.5, -4)
%      -- (4,3.5)
%	  -- (4,4)
%	  -- (3.5,4)
%	  -- (-4, -3.5)
%	  -- (-4, -4)
%	  -- (-3.5,-4)
%      -- cycle;
%\draw[scale = 1, thick,->] (-4,0) -- (4,0) node[right] {$v_1$};
%\draw[scale = 1, thick,->] (0,-4) -- (0,4) node[left] {$v_2$};
%\draw[scale = 1,domain=-4:4,smooth,variable=\x, dashed] plot ({\x},{\x});
%\draw[scale = 1,domain=-4:3.5,smooth,variable=\x] plot ({\x},{\x+0.5});
%\draw[scale = 1,domain=-3.5:4,smooth,variable=\x] plot ({\x},{\x-0.5});
%
%\draw[scale = 1,domain=-4:2,smooth,variable=\x] plot ({\x},{\x+2});
%\draw[scale = 1,domain=-2:4,smooth,variable=\x] plot ({\x},{\x-2});
%\draw	(0.5,0) node[anchor=north] {$s_0$};
%\draw	(2.1,0) node[below ] {$4s_0$};
%\draw	(2,2) node[ ] {$B_1$};
%\draw	(3.5,2.5) node[ ] {$B_2$};
%\end{tikzpicture}
%\end{center}
%\caption{Consensus manifold (dashed line) and basin of attraction in the case $d=1$, $N=2$.}
%\end{figure}

\subsubsection{Collapse prevention}\label{Sec:CollapsePrev}

We saw in Section \ref{Sec:BlackHole} that if $\lim_{s\to 0} sa(s) = +\infty$, there exists a \textit{black hole} in which no control allows to avoid clustering. 
On the other hand, we will show that if  $\lim_{s\to 0} sa(s) = 0$, then consensus can always be avoided, in particular with the sparse control $u^V$ strategy defined in Section \ref{Sec:Entropy}.

\begin{theorem}\label{th:consensusprevention}
Suppose that $\lim_{s\to 0} sa(s) = 0$. Let $M>0$. 
Let $R_i:=x_i-\bar{x}$, 
and define $i_V:= \argmax\limits_{i\in\{1,\ldots,N\}}\|R_i\|$.
Then the sparse control strategy $u^V$ defined 
 by 
\begin{equation}\label{eq:contV2}
u^V_i =  \begin{cases}
M\frac{R_{i}}{\|R_{i}\|} \; \text{ for } i=i_V \\
0 \; \text{ for all } i\neq i_V
\end{cases}
\end{equation}
prevents consensus. More specifically, there exists $\delta>0$ such that if $V(0)\leq\delta$,  
then $V(t)>V(0)$ for all $t>0$. 
\end{theorem}

\begin{proof}
Let $M>0$. Let $\epsilon = \frac{M}{8(N-1)}$. Since  $\lim_{s\to 0} sa(s) = 0$, there exists $\eta>0$ such that if for $i,j\in\{1,\ldots,N\}$, $\|x_i-x_j\|\leq \eta$, then $a(\|x_i-x_j\|)\|x_i-x_j\|\leq \epsilon$.
Suppose that $V(0)\leq \delta:= \frac{\eta^2}{2N^2}$. 
Since $V= \frac{1}{2N^2}\sum\limits_{i=1}^N\sum\limits_{j=i+1}^N\|x_i(t)-x_j(t)\|^2$, we have: 
$$
\max\limits_{(i,j)\in\{1,\ldots,N\}^2}\|x_i(0)-x_j(0)\|\leq \sqrt{2\delta}\; N = \eta.
$$
Then using the control $u^V$, we compute: 
\begin{equation*}
\begin{split}
\dot{V} 
& = - \frac{1}{N^2} \sum\limits_{i=1}^N\sum\limits_{j= i+1}^N a(\|x_i-x_j\|)\|x_i-x_j\|^2 +  \frac{1}{2N^2}\sum\limits_{j\neq i_V}\left\langle x_{i_V}-x_j, M \frac{x_{i_V}-\bar{x}}{\|x_{i_V}-\bar{x}\|}\right\rangle \\
& \geq - \frac{1}{N^2} \epsilon \sum\limits_{i=1}^N\sum\limits_{j= i+1}^N \|x_i-x_j\| +  \frac{1}{2N^2} M N \|x_{i_V}-\bar{x}\|\\
& \geq - \frac{1}{N^2} \epsilon N(N-1)\eta  +   \frac{1}{2N} M \|x_{i_V}-\bar{x}\|.
\end{split}
\end{equation*}
%where we used the inequalities $\|x_i-x_j\|\leq \eta$ and $a(\|x_i-x_j\|)\|x_i-x_j\|\leq \epsilon$ for all $i,j\in \{1,\ldots,N\}$.
%Let $\delta := \frac{1}{N(N-1)} \sum\limits_{i=1}^N\sum\limits_{j= i+1}^N \|x_i-x_j\|^2$ represent the mean of all pairwise squared distance. 
By definition of $u^V$ and $i_V$, the distance $\|x_{i_V}-\bar{x}\|$ represents the maximal distance between an agent and the average position of the group. Hence, $\eta\leq 2\|x_{i_V}-\bar{x}\|^2$. Then: 
\begin{equation*}
\begin{split}
\dot{V} 
& \geq - \frac{1}{N} \epsilon (N-1)\eta  +   \frac{1}{2N} M \frac{\eta}{2} = \frac{\eta}{N} ( -\epsilon(N-1) + \frac{M}{4}) \geq \frac{\eta M}{8N}>0 
\end{split}
\end{equation*}
by definition of $\epsilon$.
Hence $\dot V>0$ which ensures that $V(t)>V(0)$ for all $t>0$. 
\end{proof}

Theorem \ref{th:consensusprevention} states that no matter the given strength of the control $M$, if initially the system is not at consensus, it can be kept from converging to consensus. We extend this result with a more general theorem that shows that if the system starts away from the clustering set, then it can be controlled to remain bounded away from the clustering set. 

\begin{theorem}\label{th:clusteringprevention}
Let $M>0$. Suppose that $\lim_{s\to 0} sa(s) = 0$. If for every $(i,j)\in\{1,\ldots,N\}^2$, $\|x_i-x_j\|>0$, there exists a sparse feedback control strategy that prevents clustering. More specifically, 
let $\epsilon <\frac{M}{N}$ and $\delta>0$ such that $s a(s)\leq \epsilon$ for all $s\leq \delta$. Consider a function $g\in C^1((0,+\infty))$ satisfying the conditions of Definition \ref{def:Wg}, with the additional assumptions $m:=\sup g(s)>0$ and $g(\delta^2)=0$. We define the control $u^\delta$ by
\begin{equation}\label{eq:contWdelta}
u^\delta_i =  \begin{cases}
M\frac{S^\delta_{i}}{\|S^\delta_{i}\|} \; \text{ for } i=i_\delta \\
0 \; \text{ for all } i\neq i_\delta
\end{cases}
\end{equation}
where 
 $$S_i^\delta := \sum\limits_{j\neq i, \; \|x_i-x_j\|\geq \delta} g'(\|x_i-x_j\|^2)\langle x_i-x_j\rangle \quad \text{ and } \quad
i_\delta:= \arg\max\limits_{i\in\{1,\ldots,N\}}\|S^\delta_i\|.$$
Then the solution of \eqref{eq:Krause} with control $u^\delta$ satisfies
$$
\|x_i(t)-x_j(t)\|\geq \kappa := \left( g^{-1}(2N^2W_g^\delta(0)-(\frac{N(N-1)}{2}-1)m) \right)^{1/2}
$$
for all time $t\geq 0$, and for every $(i,j)\in\{1,\ldots,N\}^2$.
% exists $\kappa>0$ such that for every $(i,j)\in\{1,\ldots,N\}^2$, $\|x_i(t)-x_j(t)\|>\kappa$ for all $t>0$. 
\end{theorem}

\begin{proof}
Let $\epsilon<\frac{M}{N}$. Since $\lim_{s\to 0} sa(s) = 0$, there exists $\eta>0$ such that if $s\leq\eta$, then $sa(s)\leq \epsilon$.
Let $\delta\geq\eta$. Let $g\in C^1((0,+\infty))$ satisfy the conditions of Definition \ref{def:Wg}, with $m:=\sup g(s)$. Suppose also that $g(\delta^2)=0$. 
We define the ``danger'' set as: $\mathcal{D}^\delta = \{ (i,j)\in\{1,\ldots,N\}^2, \; | \; i<j \text{ and } \|x_i-x_j\|\leq \delta\}$.
We now define a partial generalized entropy using only the distances between pairs of agents belonging to $\mathcal{D}$: $$W_g^\delta = \frac{1}{2N^2}\sum\limits_{(i,j)\in\mathcal{D}^\delta} g(\|x_i-x_j\|^2).$$ Notice that in the case where for all $(i,j)\in\{1,\ldots,N\}^2$, $\|x_i-x_j\|\leq \delta$, then the partial generalized entropy is equal to the generalized entropy: $W_g^\delta=W_g$.
Furthermore, notice that Theorem \ref{th:Wg} can be extended to $W_g^\delta$: if $W_g^\delta\geq K$, then for all $(i,j)\in\{1,\ldots,N\}^2$,
$\|x_i-x_j\|^2\geq g^{-1}(2N^2 K-(\frac{N(N-1)}{2}-1)m)$.
Let $\mathcal{D}^\delta(0)$ and $W_g^\delta(0)$  respectively represent the danger set and the generalized entropy at $t=0$. We will prove that for all $t\geq 0$, for every $ (i,j)\in\{1,\ldots,N\}^2$, 
$$\|x_i-x_j\|\geq \sqrt{g^{-1}(2N^2W_g^\delta(0)-(\frac{N(N-1)}{2}-1)m)}. $$
The derivative of the partial generalized entropy can be written as: 
$$
\dot{W}_g^\delta = \frac{1}{N}\sum\limits_{(i,j)\in\mathcal{D}^\delta} g'(\|x_i-x_j\|^2)\langle x_i-x_j, \frac{1}{N}\sum\limits_{k\neq i} a(\|x_i-x_j\|)(x_i-x_j) +u_i\rangle.
$$
Let $S_i^\delta := \sum\limits_{j, \; (i,j)\in\mathcal{D}^\delta} g'(\|x_i-x_j\|^2)\langle x_i-x_j\rangle$. 
Let $i_\delta:= \arg\max\limits_{i\in\{1,\ldots,N\}}\|S^\delta_i\|$.
As in Proposition \ref{prop:contWg}, the control defined by
\eqref{eq:contWdelta} maximizes $W^\delta_g$ instantaneously.
With this control, the partial generalized entropy's evolution is given by:
$$
\dot{W}_g^\delta = \frac{1}{N}\sum_{i=1}^N \langle S_i^\delta, \frac{1}{N}\sum\limits_{k\neq i} a(\|x_i-x_j\|)(x_i-x_j) \rangle + \|S_{i_\delta}^\delta\| \frac{M}{N}.
$$
For all $(i,j)\in\mathcal{D}^\delta$, $\|x_i-x_j\|\leq \delta\leq \eta$, so $a(\|x_i-x_j\|)\|x_i-x_j\|\leq \epsilon$.
Let $t_\delta:=\inf\{t>0\; | \; \mathcal{D}^\delta(t)\neq \mathcal{D}^\delta(0)\}$.
While $t<t_\delta$, $W_g^\delta$ satisfies: 
$$
\dot{W}_g^\delta \geq \|S_{i_\delta}^\delta\| (-\epsilon+ \frac{M}{N}) >0.
$$
Hence for all $t<t_\delta$, $W_g^\delta(t)\geq W_g^\delta(0)$. This implies that for every $(i,j)\in\mathcal{D}^\delta$, for all $t<t_\delta$, 
\begin{equation}\label{eq:xijboundedbelow}
\|x_i(t)-x_j(t)\|\geq  \kappa:=\left( g^{-1}(2N^2W_g^\delta(0)-(\frac{N(N-1)}{2}-1)m) \right)^{1/2}.
\end{equation}
Since $\delta>\kappa$, the inequality $\|x_i(t)-x_j(t)\|\geq \kappa$ holds true for every $(i,j)\in\{1,\ldots,N\}^2$, for all $t<t_\delta$.
%At $t=t_\delta$, there exist $(i,j)$ such that either $(i,j)\in \mathcal{D}^\delta(t_\delta-\theta)$ and $(i,j)\notin \mathcal{D}^\delta(t_\delta+\theta)$, or $(i,j)\notin \mathcal{D}^\delta(t_\delta-\theta)$ and $(i,j)\in \mathcal{D}^\delta(t_\delta+\theta)$.
Let us now consider what happens when $t>t_\delta$. Notice that even though $\mathcal{D}^\delta$ changes at $t_\delta$, $W_g^\delta$ is continuous. Indeed,  
\[
\begin{split}
\lim_{\theta\rightarrow 0} & W_g^\delta(t_\delta-\theta)  = \lim_{\theta\rightarrow 0} \sum_{(i,j)\in\mathcal{D}^\delta(t_\delta-\theta)}g(\|x_i(t_\delta-\theta)-x_j(t_\delta-\theta)\|^2) \\
= & 
\lim_{\theta\rightarrow 0} \sum_{\substack{(i,j)\in\\\mathcal{D}^\delta(t_\delta-\theta)\cap\mathcal{D}^\delta(t_\delta+\theta)}}g(\|x_i(t_\delta-\theta)-x_j(t_\delta-\theta)\|^2) 
+
 \lim_{\theta\rightarrow 0} \sum_{\substack{(i,j)\in \\ \mathcal{D}^\delta(t_\delta-\theta)\setminus \mathcal{D}^\delta(t_\delta+\theta)}}g(\|x_i(t_\delta-\theta)-x_j(t_\delta-\theta)\|^2)
 \\ 
= &  \lim_{\theta\rightarrow 0} \sum_{\substack{(i,j)\in\\\mathcal{D}^\delta(t_\delta-\theta)\cap\mathcal{D}^\delta(t_\delta+\theta)}}g(\|x_i(t_\delta-\theta)-x_j(t_\delta-\theta)\|^2)
+ 0
\end{split}
\]
due to the property $g(\delta^2)=0$.
Similarly,
\[
\begin{split}
\lim_{\theta\rightarrow 0} & W_g^\delta(t_\delta+\theta)  = \lim_{\theta\rightarrow 0} \sum_{(i,j)\in\mathcal{D}^\delta(t_\delta+\theta)}g(\|x_i(t_\delta+\theta)-x_j(t_\delta+\theta)\|^2) \\
= & 
\lim_{\theta\rightarrow 0} \sum_{\substack{(i,j)\in\\\mathcal{D}^\delta(t_\delta+\theta)\cap\mathcal{D}^\delta(t_\delta-\theta)}}g(\|x_i(t_\delta+\theta)-x_j(t_\delta+\theta)\|^2)
+
\lim_{\theta\rightarrow 0} \sum_{\substack{(i,j)\in \\ \mathcal{D}^\delta(t_\delta+\theta)\setminus \mathcal{D}^\delta(t_\delta-\theta)}}g(\|x_i(t_\delta+\theta)-x_j(t_\delta+\theta)\|^2)
  \\ 
= &  \lim_{\theta\rightarrow 0} \sum_{\substack{(i,j)\in\\\mathcal{D}^\delta(t_\delta+\theta)\cap\mathcal{D}^\delta(t_\delta-\theta)}}g(\|x_i(t_\delta+\theta)-x_j(t_\delta+\theta)\|^2)
+ 0.
\end{split}
\]
We can now apply the same reasoning as previously starting at $t=t_\delta$. Let $t_\delta':=\inf\{t>t_\delta\; | \; \mathcal{D}^\delta(t)\neq \mathcal{D}^\delta(t_\delta)\}$. We show that equation \eqref{eq:xijboundedbelow} becomes: for all $t_\delta<t<t_\delta'$, for every $(i,j)\in \{1,\ldots,N\}^2$,
\[
\|x_i(t)-x_j(t)\| \geq  \left( g^{-1}(2N^2 W_g^\delta(t_\delta)-m(\frac{N(N-1)}{2}-1)) \right)^{1/2} = \kappa.
\] 
By induction, we obtain:  $\|x_i(t)-x_j(t)\|\geq \kappa$ for all $t\geq 0$, for every $(i,j)\in \{1,\ldots,N\}^2$.

\end{proof}

\section{Macroscopic model (kinetic equation)}\label{Sec:kinetic}

The second part of this paper focuses on the kinetic limit of system \eqref{eq:gensyst} and establishes the kinetic version of the results presented in Section \ref{Sec:finite}.

For every time $t\geq 0$, let $\mu(t,\cdot)\in\P(\R^d)$ be a probability measure representing the density of agents at time $t$.  
We suppose that we are allowed to control a part of the state space denoted by $\omega\subset \R^d$. Denoting by $\chi_\omega$ the characteristic function of $\omega$, we write the control as: $\chi_\omega u$, with $u:\R\times\R^d\rightarrow \R^d$ representing the control strength.
%Given $M$, $c>0$, we set the constraints $\omega\in\Omega_c$ and $u\in\UMkin$, where:
%\begin{equation*}
%\begin{cases} 
%\UMkin = \{ u : \R\times\R^d\rightarrow \R^d \text{ measurable, } \|u(t,\cdot)\|_{L^\infty(\R^d)}\leq M \}\\
%\Omega_c = \{\omega \subset \R^d\; | \; \int_\omega dx\leq c\}. 
%\end{cases}
%\end{equation*}
Given $M$, $c>0$, We set the following constraints on $\omega$ and $u$:
\begin{equation}\label{eq:condcontkin}
\begin{cases} 
u\in\UMkin := \{ u : \RR\times\RR^d\rightarrow \RR^d \text{ measurable } \; | \; \|u(t,\cdot)\|_{L^\infty(\RR^d)}\leq M \}\\
 \omega \in\Omega_c := \{\omega \subset \RR^d\; | \; |\omega| \leq c\}
\end{cases}
\end{equation}
where $|\omega| = \int_\omega dx$ is the Lebesgue measure of $\omega$.
As in \cite{PRT15}, the condition $|\omega| \leq c$ allows us to extend the idea of sparse control to the kinetic setting. Instead of acting on a single agent as in the discrete case, we limit the size of the state space region that the control can act on. Another possibility, not explored in this paper, would be to limit the mass of agents that can be controlled, with a condition such as $\int_\omega d\mu_x\leq c<1$, as in \cite{ACMPPRT17, PRT15}.
Let $\xi$ denote the interaction kernel. In the case of the Hegselmann-Krause interaction, 
$$
\xi[\mu](x) = \int_{\R^d} a(\|x-y\|)(y-x)d\mu(y).
$$
Then the kinetic version of \eqref{eq:gensyst} can be written as:
\begin{equation}\label{eq:kindym}
\partial_t\mu + \div((\xi[\mu]+\chi_\omega u)\mu) = 0.
\end{equation}

The well-posedness of the kinetic equation \eqref{eq:kindym} has been established in the case where $a(\cdot)$ is a Lipschitz function (see \cite{PRT15}). Here, we want to allow possible blow-ups of the interaction function $a$, which in turn causes finite-time blow-up of initially regular solutions. The existence and uniqueness of weak measure solutions for kinetic equations with aggregation phenomena was studied in a number of papers, for various classes of interaction potentials, see for instance \cite{CDFLS11} where well-posedness was proven in the case of semi-convex potentials, and \cite{BLR11}, where the authors establish an $L^p$ theory for the multi-dimensional aggregation equation (and all the references within). 

\subsection{Kinetic generalized entropy functional}

\subsubsection{Controlling the kinetic system}

From here onward we will suppose that $\mu$ is of compact support, \textit{i.e.}, $\mu\in\P_c(\R^d)$.
We define consensus in the context of the kinetic formulation. Let $\delta_x$ denote the Dirac mass centered at $x$.
\begin{definition}\label{def:consensuskin}
For any $x_0\in\R^d$, the state $\mu=\delta_{x_0}$ is referred to as \textbf{consensus}.
\end{definition}
As in the discrete case, we
 define the variance of the system which characterizes consensus: 
$$
\V(\mu(t,\cdot)) = \iint \|x-y\|^2 d\mu_x(t)d\mu_y(t).
$$
\begin{lemma}\label{lemma:Vkin}
$\V(\mu)=0$ if and only if $\mu=\delta_{x_0}$ for some $x_0\in\R^d$.
\end{lemma}
\begin{proof}
Suppose that there exists $x_0\in\R^d$ such that $\mu = \delta_{x_0}$. Then clearly
$\V(\delta_{x_0}) = 0$.
Conversely, suppose that there exist $(x_1,x_2)\in(\supp(\mu))^2$ such that $\|x_1- x_2\|=\delta>0$. Then $\mu(B(x_1,\frac{\delta}{3}))>0$ and $\mu(B(x_2,\frac{\delta}{3}))>0$ (where $B(x,r)$ denotes the ball of center $x$ and radius $r$). We can write:
$$
\V(\mu) \geq \int_{B(x_1,\frac{\delta}{3})}\int_{B(x_2,\frac{\delta}{3})} \|x-y\|^2 d\mu_xd\mu_y \geq \left(\frac{\delta}{3}\right)^2 \mu(B(x_1,\frac{\delta}{3})) \mu(B(x_2,\frac{\delta}{3})) >0.
$$
\end{proof}

Going further, if the variance is small, we can indeed ensure that the system is concentrated around its center of mass in the following sense:
\begin{lemma}\label{lemma:Vkinmin}
For all $r>0$, $\iint_{\|x-y\|\geq r} d\mu_x d\mu_y \leq \frac{1}{r^2}\V(\mu)$.
\end{lemma}
\begin{proof}
It suffices to observe that for all $r>0$,
\begin{equation*}
\begin{split}
\V \geq \iint_{\|x-y\|\geq r} \|x-y\|^2 d\mu_x d\mu_y 
& \geq r^2 \iint_{\|x-y\|\geq r} d\mu_x d\mu_y
\end{split}
\end{equation*}
and the claim follows.
\end{proof}
Then it is clear that maximizing $\V$ ensures that consensus is avoided. Like in Section \ref{Sec:Entropy}, we design a feedback control strategy maximizing the time derivative of the variance. 

\begin{prop}\label{prop:contVkin}
Let $\bar{x}(t):=\int_{\R^d}xd\mu_x(t)$ denote the center of mass of $\mu(t,\cdot)\in\P_c(\R^d)$.
Let $\Rkin(x,t) := \int (x-y) d\mu_y(t) = x-\bx(t)$.
The control $\chi_\omega u$ such that 
\begin{equation}\label{eq:contVkin}
\omega(t):= \argmax_{\omega\in\Omega_c} \int_\omega \|\Rkin(x,t)\| d\mu_x(t); \quad 
u(x,t) = M\frac{\Rkin(x,t)}{\|\Rkin(x,t)\|} \text{ for all } x\in \omega\setminus\{\bx\}
\end{equation}
maximizes $\dot{V}$ instantaneously with the constraints \eqref{eq:condcontkin}.
\end{prop}
\begin{proof}
We calculate the time derivative of the kinetic entropy: 
\begin{equation*}
\begin{split}
\DDt\V(\mu(t,\cdot)) & = 2 \iint \|x-y\|^2 d\left( -\div_x((\xi[\mu](x,t)+\chi_{\omega(t)} u(t,x))\mu_x(t)\right) d\mu_y(t) \\
 & = 2 \iint \nabla_x(\|x-y\|^2)\cdot(\xi[\mu](x,t)+\chi_{\omega(t)} u(t,x))d\mu_x(t) d\mu_y(t) \\
 & = 4 \iint (x-y) \cdot\xi[\mu](x,t)d\mu_x(t) d\mu_y(t)+ 4 \int (x-\bx(t)) \cdot\chi_{\omega(t)} u(t,x)d\mu_x(t).
\end{split}
\end{equation*}
Hence the control given by \eqref{eq:contVkin} maximizes $\dot \V$.
\end{proof}

As for the finite-dimensional model of Section \ref{Sec:finite}, we aim not only to prevent convergence of the system to consensus, but also any clustering. 
We define kinetic clustering as follows:
\begin{definition}\label{def:clusterkin}
If $\mu\in\P(\R^d)$ contains at least one point mass,
%can be decomposed as $\mu = \alpha \delta_x + \tilde\mu$, with $\alpha>0$, $x\in\R^d$ and $\tilde\mu$ a positive measure such that $\tilde{\mu}(\R^d)+\alpha = 1$,
we say that $\mu$ is in clustered state. 
\end{definition} 
\begin{rem}
In Definition \ref{def:clusterkin}, the number of Dirac masses contained in $\mu$ represents the number of clusters. The absolutely continuous part of $\mu$ represents the non-clustered agents. Notice that consensus is a special case of clustering, with $\tilde\mu = 0$.
\end{rem}
As in the discrete case, the strict positivity of $\V$ is not enough to ensure avoidance of clustering.
Indeed, Let $\mu=\frac{1}{2}(\delta_{x_1}+\delta_{x_2})$, with $(x_1,x_2)\in(\R^d)^2$, $x_1\neq x_2$. Then $\mu$ is in a clustered state and yet $\V(\mu)=\frac{1}{2}\|x_1-x_2\|^2>0$. 
As in Section \ref{Sec:Entropy}, we define the kinetic generalized entropy: given a function $g\in C^1(\R^d)$ satisfying the hypotheses of Definition \ref{def:Wg}, 
$$
\Wg(\mu(t,\cdot)) = \iint g(\|x-y\|^2) d\mu_x(t)d\mu_y(t).
$$
The kinetic counterpart of Theorem \ref{th:Wg} shows that maximizing the kinetic generalized entropy prevents clustering. 
%\begin{lemma}\label{lemma:Wgminkin}
%Let $\mu=\sum_{k\in\Gamma} \alpha_k \delta_{x_k} + \tilde{\mu}$ be a clustered state defined as in Definition \ref{def:clusterkin}. Then $W_g(\mu)=-\infty$.
%\end{lemma}
%\begin{proof}
%Notice that 
%\begin{equation*}
%\begin{split}
%& W_g(\mu) = \sum_{k\in\Gamma} \alpha_k^2 g(\|x_k-x_k\|^2) +  \sum_{i,j\in\Gamma, i\neq j}\alpha_i\alpha_j  g(\|x_i-x_j\|^2) + 
% \\ + & 2\sum_{k\in\Gamma} \alpha_k \int g(x_k-y)d\tilde{\mu}_y + \iint  g(\|x-y\|^2) d\tilde\mu_x(t)d\tilde\mu_y(t)\\
%& \leq \sum_{k\in\Gamma} \alpha_k^2 g(0) + m (\sum_{i,j\in\Gamma, i\neq j}\alpha_i\alpha_j + 2  \sum_{k\in\Gamma} \alpha_k +1) = -\infty.
%\end{split} 
%\end{equation*}
%\end{proof}
\begin{lemma}\label{lemma:mumin}
Suppose that $\Wg(\mu)\geq K \in\R$ and $\sup_{s>0} g(s)=m$. Then for all $r$ small enough, 
\begin{equation}\label{eq:lemmakin}
\iint_{\|x-y\|\leq r} d\mu_x d\mu_y \leq \frac{m-K}{-g(r^2)}.
\end{equation}
\end{lemma}
\begin{proof} 
We have:
\begin{equation*}
\Wg = \iint\limits_{\|x-y\|\leq r}g(\|x-y\|^2) d\mu_x d\mu_y+\iint\limits_{\|x-y\|> r}g(\|x-y\|^2) d\mu_x d\mu_y  \leq g(r^2)\iint_{\|x-y\|\leq r} d\mu_x d\mu_y + m.
\end{equation*}
For $r$ small enough, $g(r^2)<0$, so dividing by $g(r^2)$ yields \eqref{eq:lemmakin}.
\end{proof}
Lemma \ref{lemma:mumin} implies that if $\Wg>K>0$, $\lim_{r\rightarrow 0}\iint_{\|x-y\|\leq r} d\mu_x d\mu_y=0$. Hence $\mu$ cannot be in a clustered state.
Indeed, if $\mu = \alpha \delta_{x_0} + \tilde{\mu}$, we have $ \iint_{\|x-y\|=0} d\mu_x d\mu_y \geq \iint_{\|x-y\|=0} \alpha^2 d\delta_{x_0} d\delta_{x_0} = \alpha^2>0$. 
%Lemmas \ref{lemma:Wgminkin} and \ref{lemma:mumin} give the following result:
%\begin{theorem}
%$W_g(\mu)\geq K>0$ if and only if the system is not in a clustered state.
%\end{theorem} 
%\begin{proof}
%One implication is given 
%\end{proof}

As in Section \ref{Sec:Entropy}, we design
a control strategy maximizing $\dot{\Wg}$, in order to steer the system away from clustering.
%\begin{prop}
%Let $u\in\UMkin$ and $\omega\in\Omega_c$.
%Let $\Rkin(x) = \int (x-y)d\mu_y(t)$ and $\SS(x) = \int g'(\|x-y\|^2)(x-y)d\mu_y(t)$.
%The controls $\chi_{\omega_V} u^V$ and $\chi_{\omega_{\Wg}} u^{\Wg}$ respectively maximize $\dot{\V}$ and  $\dot{\Wg}$ instantaneously, where:
%\begin{equation}
%\begin{cases}
%\omega_\Phi:= \argmax_{\omega\in\Omega_c} \int_\omega \| \Phi(x) \| d\mu_x(t) \\
%u^\Phi = M\frac{\Phi}{\|\Phi\|}.
%\end{cases}
%\end{equation}
%\end{prop}

\begin{prop}\label{prop:Wgkin}
Consider $\mu\in\P_c(\R^d)$ and its kinetic generalized entropy $\Skin(x) = \int g'(\|x-y\|^2)(x-y)d\mu(t,y)$.
Let $M,c>0$ and let $u$ and $\omega$ satisfying the conditions \eqref{eq:condcontkin}.
The control $\chi_\omega u$ such that 
\begin{equation}\label{eq:contWkin}
\omega(t):= \argmax_{\omega\in\Omega_c} \int_\omega \Skin(x,t) d\mu(t,x); \quad
u(x,t) = M\frac{\Skin(x,t)}{\|\Skin(x,t)\|}
\end{equation}
maximizes $\dot{\Wg}$ instantaneously.
\end{prop}
\begin{proof}
 The time derivative of $W_g$ can be computed as:
\begin{equation*}
\begin{split}
\frac{d}{dt}\Wg(t) & = 2 \iint g(\|x-y\|^2) d(-\div(( \xi[\mu](x,t) + \chi_\omega(t) u(x,t))\mu_x(t)))d\mu_y(t) \\
& = 2 \iint \nabla_x g(\|x-y\|^2) \cdot ( \xi[\mu](x,t) + \chi_\omega(t) u(x,t)) d\mu_x(t)d\mu_y(t) \\
& = 4 \int \left(\int g'(\|x-y\|^2) (x-y) d\mu_y(t) \right)\cdot ( \xi[\mu](x,t) + \chi_\omega(t) u(x,t)) d\mu_x(t).
\end{split}
\end{equation*}
Let $\Skin(x,t):= \int g'(\|x-y\|^2) (x-y) d\mu_y(t) $. Then 
\begin{equation*}
\frac{d}{dt}\Wg(t) = 2 \int_{\RR^d} \Skin(x,t) \cdot X[\mu](x) d\mu_x(t) + 2 \int_\omega \Skin(x,t) \cdot u \; d\mu_x(t).
\end{equation*}
Hence the derivative of $\Wg$ is maximized by the control $\chi_\omega u$ defined by \eqref{eq:contWkin}.
\end{proof}

\subsubsection{Behavior of the kinetic Hegselmann-Krause system without control}

Before studying the controlled system, we look at its behavior without control in two specific cases.
Firstly, we can prove that if $a(\cdot)$ is bounded below, then the system tends to consensus exponentially.
Secondly, if $s\mapsto s a(s)$ is bounded at $0$, and if $\mu_0$ is of compact support, then the system tends to consensus in finite time.

\begin{theorem}\label{th:Vkinconvergence}
Let $\mu$ satisfy the kinetic Hegselmann-Krause system (without control)
\begin{equation}\label{eq:kinsystnocontrol}
\begin{cases}
\partial_t\mu + \div(\xi[\mu]\mu) = 0\\
\mu(0) = \mu_0
\end{cases}
\end{equation}
where the convolution kernel is $\xi[\mu](x) = \int_{\R^d} a(\|x-y\|)(y-x)d\mu(y)$, for a  continuous function $a\in C(\R^+,\R^+)$.
If $a$ is bounded away from zero, then $\V(\mu(t,\cdot))$ tends to zero exponentially.
\end{theorem}
\begin{proof}
We compute the time-derivative of the kinetic variance:
\begin{equation*}
\begin{split}
\DDt \V(\mu) & =  \iint \|x-y\|^2 d(-\div_x(\xi[\mu_x]\mu_x))d\mu_y + \iint \|x-y\|^2 d\mu_x d(-\div_y(\xi[\mu_y]\mu_y)) \\
& = 2 \iint (x-y) \xi[\mu_x] d\mu_x d\mu_y + 2 \iint (y-x) \xi[\mu_y] d\mu_y d\mu_y \\
& = 2 \iiint (x-z) \cdot a(\|x-y\|)(y-x)d\mu_y d\mu_x d\mu_z + 2 \iint (y-z)\cdot a(\|y-x\|)(x-y) d\mu_x d\mu_y d\mu_z \\
& = - 2 \iint\|x-y\|^2 a(\|x-y\|) d\mu_x d\mu_y.
\end{split}
\end{equation*}
If there exists $C>0$ such that $a(s)>C$ for all $s\in\R^+$, then
$
\DDt \V \leq -2 C \V.
$
and thus $\V(\mu(t,\cdot))\leq \V(\mu_0)\exp(-2Ct)$.
\end{proof}

We now look at the behavior of measures with initially compact support, \textit{i.e.}, $\mu_0\in\P_c(\R^d)$. As in \cite{PRT15}, we define the size $X$ of the support of $\mu$ as the radius of the smallest ball centered at $\bx$ and containing the support of $\mu$. More precisely, given a time-evolving measure $\mu(t)\in\P_c(\R^d)$,
\begin{equation}\label{eq:X}
X(t) = \inf \{ X\geq 0 \; | \; \supp(\mu(t))\subseteq B(\bx(t),X)\}.
\end{equation}

\begin{theorem}\label{th:X}
Let $\mu_0\in\P_c(\R^d)$ and suppose that $\mu(t)$ is the solution of \eqref{eq:kinsystnocontrol}. 
Suppose that $a\in\Lip(\R^+,\R^+)$  and that
%Suppose that $s\mapsto sa(s)$ is bounded in a neighborhood of zero and that 
%$s\mapsto a(s)$ is non-increasing.
$s\mapsto a(s)$ is bounded below by $C>0$.
Let $X(t)$ represent the size of the support of $\mu(t)$, as given by \eqref{eq:X}. Then $X$ converges exponentially to zero.
\end{theorem}
\begin{proof}
The proof follows the same argument as that done in \cite{PRT15}. 
Since $a\in\Lip(\R^+,\R^+)$, while $\mu(t)$ has compact support, the displacement of the support has bounded velocity $\xi[\mu(t)] = \iint_{\supp(\mu(t))}\|x-y\|^2a(\|x-y\|) d\mu_x(t) d\mu_y(t)<\infty$, so $X(\cdot)$ is differentiable almost everywhere.
Let $x(\cdot,x_0)$ be the particle trajectory of \eqref{eq:kinsystnocontrol} with $x(0,x_0)=x_0$.
Hence $x(t,x_0)\in\supp(\mu(t)$ for all $x_0\in\supp(\mu_0)$, and
$$
X(t) = \max\{\|x(t,x_0)-\bx(t)\|\; | \; x_0\in\supp(\mu_0)\}.
$$ 
We remark that the maximum is reached since $\mu_0$ is assumed to have compact support.
For every $t\geq 0$, let $K_t$ denote the set of points $x_0\in\supp(\mu_0)$ that attain the maximum $X(t)$, \textit{i.e.},
$$
K_t = \argmax\{\|x(t,x_0)-\bx(t)\|\; | \; x_0\in\supp(\mu_0)\}.
$$ 
By definition, $X(t)^2 = \|x(t,x_0)-\bx(t)\|^2$ for all $x_0\in K_t$. 
From the Danskin theorem (see \cite{Danskin67}),
$$
\DDt (X(t)^2) = \max\limits_{x_0\in K_t}\{ \DDt(\|x(t,x_0)-\bx(t)\|^2)\}.
$$
Then 
\begin{equation}
\begin{split}
\dot{X}(t) X(t) & = \max\limits_{x_0\in K_t}\left\{ (x(t,x_0)-\bx(t))\cdot \int a(\|x(t,x_0)-y\|)(y-x(t,x_0)) d\mu_y \right\} \\
& = \max\limits_{x_0\in K_t}\left\{\int_{\supp(\mu(t))} a(\|x(t,x_0)-y\|)(x(t,x_0)-\bx(t))\cdot (y-x(t,x_0)) d\mu_y \right\} .
\end{split}
\end{equation}
Since $x(t,x_0)\in S(\bx(t),X(t))$, by convexity, $(x(t,x_0)-\bx(t))\cdot (y-x(t,x_0)) \leq 0$ for all $y\in\supp(\mu(t))\subseteq B(\bx(t),X(t))$.
We supposed that 
%$s\mapsto a(s)$ is non increasing. 
$a(s)\geq C>0$. 
Hence 
%\begin{equation*}
%\begin{split}
%\dot{X}(t) X(t) & \leq \max\limits_{x_0\in K_t}\left\{\int_{\supp(\mu(t)} a(2X(t))(x(t,x_0)-\bx(t))\cdot (y-x(t,x_0)) d\mu_y \right\} \\
%& \leq \max\limits_{x_0\in K_t}\left\{- a(2X(t))\|x(t,x_0)-\bx(t)\|^2 \right\} = -a(2X(t)) X^2(t)
%\end{split}
%\end{equation*}
%which gives: $\dot{X}(t) \leq -a(2X(t)) X(t)$.
%Consequently, $X(t)$ is non increasing, which implies that $a(2X(t))\geq a(2X(0))$ for all $t\geq 0$.
\begin{equation*}
\begin{split}
\dot{X}(t) X(t) & \leq \max\limits_{x_0\in K_t}\left\{\int_{\supp(\mu(t)} C (x(t,x_0)-\bx(t))\cdot (y-x(t,x_0)) d\mu_y \right\} \\
& \leq \max\limits_{x_0\in K_t}\left\{-C\|x(t,x_0)-\bx(t)\|^2 \right\} = -C X^2(t).
\end{split}
\end{equation*}
and thus $ X(t) \leq X(0) \exp(-C t)$.
\end{proof}

\begin{rem}
In the statement of Theorem \ref{th:X}, the condition $s\mapsto a(s)$ bounded below by $C>0$ can be replaced with 
$s\mapsto a(s)$ non-increasing (see \cite{PRT15}), and the same conclusion holds: $X$ tends to zero exponentially.  
\end{rem}
With Theorems \ref{th:Vkinconvergence} and \ref{th:X}, we showed that under certain conditions on $a$, the system converges to consensus. We now examine under what conditions convergence to consensus can be avoided, and even further, under what conditions clustering can be avoided.

\subsection{Control of the kinetic dynamics}

We provide equivalent results to those of Section \ref{Sec:finite}, but in the case of the kinetic dynamics.
We adapt the concepts of \textit{black hole}, \textit{safety region}, \textit{basin of attraction} and \textit{collapse prevention} to the kinetic formulation.
Let $\mu_0(x)=\mu(0,x),$ and $\phi^t_{a,u,\omega}\#\mu_0$ denote the push-forward at time $t$ of $\mu_0$ by the flow of the kinetic dynamics \eqref{eq:kindym} with interaction potential $a$, and control $u\chi_\omega$.
\subsubsection{Black hole} As in the discrete case, we start by defining the \textit{black hole} region
$$
\mathcal{R}^M_\mathrm{BH} = \{\mu_0\in\P(\R^d)^n\; | \; \forall u, \; \omega, \; \exists T>0, \; \V(\phi^T_{a,u,\omega}\#\mu_0)=0 \}.
$$

We begin by pointing out the following:
\begin{lemma}\label{lemma:Vkinmin2}
For all $r>0$, $\iint_{\|x-y\|\geq r}\|x-y\| d\mu_x d\mu_y \leq \frac{1}{r}\V(\mu)$.
\end{lemma}
\begin{proof}
The proof follows the same argument as that of Lemma \ref{lemma:Vkinmin}.
\end{proof}

\begin{theorem}\label{th:BHkin}
Suppose that
$\lim_{s\rightarrow 0} sa(s) = + \infty  $. Then for all $M>0$, there exists a \textbf{black hole region} $\mathcal{R}^M_\mathrm{BH}\neq \emptyset$.
More specifically, for all $\mu_0\in\mathcal{R}^M_\mathrm{BH}$, $\V(\phi^t_{a,u,\omega}\#\mu_0)$ tends to $0$ in finite time, and this for any control $u\in\UMkin$, $\omega\in\Omega_c$. 
\end{theorem}
\begin{proof}
Let $A=2M$. There exists $r_0>0$ such that $sa(s) \geq A$ for all $s\geq r_0$. Suppose that $\V(0)\leq \frac{r_0^2}{4}$.
\begin{equation*}
\begin{split}
\dot \V & \leq -2\iint \|x-y\|^2 a(\|x-y\|) d\mu_x d\mu_y + M \iint \|x-y\| \chi_\omega(x) d\mu_x d\mu_y \\
& \leq -2\iint_{\|x-y\|\leq r_0} \|x-y\|^2 a(\|x-y\|) d\mu_x d\mu_y + M \iint \|x-y\|  d\mu_x d\mu_y \\
& \leq -2A \iint_{\|x-y\|\leq r_0} \|x-y\| d\mu_x d\mu_y + M \iint \|x-y\|  d\mu_x d\mu_y  \\
& \leq -2A \iint  \|x-y\| d\mu_x d\mu_y +2 A  \iint_{\|x-y\| > r_0} \|x-y\| d\mu_x d\mu_y + M \iint \|x-y\|  d\mu_x d\mu_y  \\
& \leq (M-2A)\iint\|x-y\|d\mu_x d\mu_y + 2A\frac{\V}{r_0}
\end{split}
\end{equation*} 
where the last inequality is a consequence of Lemma \ref{lemma:Vkinmin2}.
Notice that 
$$
\iint \|x-y\| d\mu_x d\mu_y \leq \left( \iint \|x-y\|^2 d\mu_x d\mu_y\right)^{1/2}\left(\iint 1 d\mu_x d\mu_y\right)^{1/2} = \sqrt{\V}.
$$
Then we have 
$$
\dot \V \leq (M-2A+\frac{2A}{r_0}\sqrt{\V})\sqrt{\V}.
$$
While $\sqrt{\V}\leq \frac{r_0}{2}$, $\dot \V\leq -M\sqrt{\V}$. Hence $\V$ decreases which ensures that the condition  $\sqrt{\V}\leq \frac{r_0}{2}$ holds. In conclusion, $\V$ tends to $0$ in finite time.
\end{proof}

%We now adapt the results of Sections \ref{Sec:Safety}, \ref{Sec:basin} and \ref{Sec:CollapsePrev} to the kinetic setting. For conciseness, we only state the results and omit the proofs.

\subsubsection{Safety region}

The behavior of the interaction function at infinity determines the existence of either a \textit{safety region} or a \textit{basin of attraction}. For the discrete system, the \textit{safety region} was defined by bounding below the smallest pairwise distance $\min_{i\neq j}\|x_i-x_j\|$. 
In the kinetic case, we replace this condition by requiring that the population density stays split into distinct measures concentrated around points that are far enough from one another.
%we require that $\mu$ is not too concentrated, that is that the quantity $\iint_{\|x-y\|\leq r}d\mu_x d\mu_y$ is bounded below. 
%Then we define the kinetic \textbf{safety region} as:
%$$
%\mathcal{R}^M_\mathrm{S} = \{\mu_0\in\P(\R^d) | \; \exists u, \; \omega, \; \exists K,\; \forall t, \; \Wg(\phi^t_{a,u}\#\mu_0)\geq K \}.
%$$
More precisely, when the interaction function $a(\cdot)$ decreases enough at infinity, we can prove the following.
\begin{theorem}\label{th:safetykin}
Suppose that $\lim_{s\rightarrow+\infty} s a(s) = 0$, and that $a\in\Lip(\R^+,\R^+)$.
Let $N\in\N$, $N\geq 2$, and $r>0$ small enough that $N |B(0,r)|\leq c$. Let $R>2r$ be large enough that for all $s\geq R-2r$, $sa(s)\leq\epsilon < M$. Let $\{x_1,\ldots, x_N\} \in (\R^d)^N$ be such that for all $i\neq j$, $\|x_i-x_j\|\geq R$.
If $\mu_0$ satisfies: 
\begin{equation}\label{eq:suppmu0}
\supp(\mu_0)\subset \bigcup_{i=1}^N B(x_i, r),
\end{equation}
then for $\omega:=\bigcup_{i=1}^N B(x_i, r)\subset\Omega_c$ and $u\in\UMkin$ defined by 
\begin{equation}\label{eq:contsuppmu0}
u(t,x) = 
\begin{cases}
-M\frac{x-x_i}{\|x-x_i\|} \text{ for all } x\in B(x_i,r)\setminus \{x_i\}, \text{ for all } i\in\{1,\ldots,N\} \\
0 \text{ otherwise,}
\end{cases}
\end{equation}
%and satisfying $u\in\UM$ and $\omega\in\Omega_c$, and  such that for all $t\geq 0$, 
the solution to the controlled kinetic equation \eqref{eq:kindym} satisfies for all $t\geq 0$:
$$\supp(\mu(t,\cdot))\subset \bigcup_{i=1}^N B(x_i, r).$$ 
\end{theorem}
\begin{proof}
Suppose that $\mu_0\in\P_c(\R^d)$ and satisfies \eqref{eq:suppmu0}, 
with $N$, $r$, $R$ and $\{x_1,\ldots, x_N\} \in (\R^d)^N$ satisfying the conditions above.
As in the proof of Theorem \ref{th:X}, we denote by $x(t,x_0)$ the particle trajectory of \eqref{eq:kinsystnocontrol} with $x(0,x_0)=x_0$.
Let $X_i(t) := \max\{\|x(t,x_0)-x_i\|\; , \; x_0\in B(x_i,r)\cap \supp(\mu_0)\}$.
Let $K^i_t = \argmax\{\|x(t,x_0)-x_i\|\; | \; x_0\in B(x_i,r)\cap\supp(\mu_0)\}$.
Then for all $t\geq 0$, for all $x_0\in K^i_t$, $X_i(t)^2 = \| x(t,x_0)-x_i\|^2$.
From Danskin's theorem, 
$$
\DDt(X_i(t)^2) = \max_{x_0\in K_t^i} \{ \DDt \|x(t,x_0)-x_i\|^2\}.
$$
We have
\begin{equation*}
\begin{split}
\DDt  \|x(t,x_0)-x_i\|^2 = & 2(x(t,x_0)-x_i)\cdot \left(\int a(\|x(t,x_0)-y\|)(y-x(t,x_0)) d\mu_y + \chi_\omega u(t,x(t,x_0))\right) \\
 = & 2(x(t,x_0)-x_i)\cdot \big( \int_{B(x_i,r)} a(\|x(t,x_0)-y\|)(y-x(t,x_0)) d\mu_y \\ & + \int_{\R^d\setminus B(x_i,r)} a(\|x(t,x_0)-y\|)(y-x(t,x_0)) d\mu_y + \chi_\omega u(t,x(t,x_0))\big).
\end{split}
\end{equation*}
% \leq & 2 (\int_{B(x_i,r)} a(\|x(t,x_0)-y\|)(x(t,x_0)-x_i)\cdot(y-x(t,x_0)) d\mu_y + \|x(t,x_0)-x_i\| \epsilon\int_{\R^d\setminus B(x_i,r)}  d\mu_y + \chi_\omega u(t,x(t,x_0)) )
Notice that for all $y\in B(x_i,r)$, $(x(t,x_0)-x_i)\cdot(y-x(t,x_0))\leq 0$. 
On the other hand, while $\supp(\mu(t,\cdot))\subset \bigcup_{i=1}^N B(x_i, r)$, for all $y\in (\R^d\setminus B(x_i,r))\cap \supp(\mu(t,\cdot))$, $\|x(t,x_0)-y\|\geq R-2r$, so $a(\|x(t,x_0)-y\|)(y-x(t,x_0))\leq \epsilon$.
Then we can write:
$$
\DDt  \|x(t,x_0)-x_i\|^2 \leq 2 (\|x(t,x_0)-x_i\| \epsilon + (x(t,x_0)-x_i)\chi_\omega u(t,x(t,x_0)) ).
$$
We design the control $u$ by \eqref{eq:contsuppmu0}
and the control set by $\omega:=\bigcup_{i=1}^N B(x_i, r)\subset\Omega_c$.
With this control, 
$$
\DDt(X_i(t)^2) = 2 X_i(t) \dot X_i(t) \leq 2 ( X_i(t)\epsilon - M X_i(t) )
$$
from which we get: $\dot X_i(t) <0$. Hence with the designed control, $\mu$ satisfies: $\supp(\mu(t,\cdot))\subset \bigcup_{i=1}^N B(x_i, r)$ for all $t\geq 0$.
\end{proof}
\begin{rem}
Theorem \ref{th:safetykin} adapts the results obtained in the microscopic setting to the kinetic case, by ensuring that the population density stays confined to balls whose centers are far apart, which prevents consensus. However, this does not prevent the measure from converging to a clustering state, which can happen if the radii of the balls containing its support converge to zero.
\end{rem}
\subsubsection{Basin of attraction}
We showed that there exists a \textit{safety region} if $a$ decreases fast enough at infinity. On the other hand, when $\lim_{s\rightarrow +\infty} sa(s) = +\infty $, no control can prevent the convergence of $\mu$ to an attractive region that we name \textit{basin of attraction}. In the discrete case, the basin of attraction consists of all states in which at least one pairwise distance is small (see Section \ref{Sec:basin}). In the kinetic setting, the basin of attraction consists of measures with a large concentration around their center of mass \eqref{eq:basinkin}.
\begin{theorem}
Suppose that
$\lim_{s\rightarrow +\infty} sa(s) = +\infty $. Then there exists a diameter $d>0$, a constant $\delta\in (0,\frac{1}{2})$ and a time $T>0$ such that for any $\mu_0\in\P(\R^d)$, for any control $\chi_\omega u$,
\begin{equation}\label{eq:basinkin}
\iint_{\|x-y\|\leq d} d\mu_x(T) d\mu_y(T) \geq \delta.
\end{equation}
\end{theorem}
\begin{proof}
%Since $\lim_{s\rightarrow \infty} sa(s) = \infty $, for all $A>0$, there exists $s_0>0$ such that for all $s\geq s_0$, $sa(s)\geq A$. Let $A=M$ and l
Let $d>0$ such that for all $s\geq d$, $sa(s)\geq M$.
We reason by contradiction. Let $\delta\in (0,\frac{1}{2})$ and suppose that for all $t>0$, 
\begin{equation}\label{eq:hypproofbasin}
\iint_{\|x-y\|\leq d} d\mu_x(t) d\mu_y(t) < \delta.
\end{equation}
The derivative of the variance was computed earlier and can be written as:
\begin{equation*}
\begin{split}
\dot \V & \leq -2\iint \|x-y\|^2 a(\|x-y\|) d\mu_x d\mu_y + M \iint \|x-y\| \chi_\omega(x) d\mu_x d\mu_y \\
& \leq -2\iint_{\|x-y\|> d} \|x-y\|^2 a(\|x-y\|) d\mu_x d\mu_y + M \iint \|x-y\| d\mu_x d\mu_y \\
& \leq -2 M \iint_{\|x-y\|> d} \|x-y\| d\mu_x d\mu_y + M \iint_{\|x-y\|\geq d} \|x-y\| d\mu_x d\mu_y
+ M \iint_{\|x-y\|< d} \|x-y\| d\mu_x d\mu_y \\
& \leq - M \iint_{\|x-y\|> d} \|x-y\|d\mu_x d\mu_y + M  \iint_{\|x-y\|\leq d}  \|x-y\| d\mu_x d\mu_y \\
& \leq - M d\iint_{\|x-y\|> d} d\mu_x d\mu_y + M d  \iint_{\|x-y\|\leq d} d\mu_x d\mu_y .
\end{split}
\end{equation*} 
From \eqref{eq:hypproofbasin}, we also have $- \iint_{\|x-y\|> d} d\mu_x d\mu_y <\delta -1 $. We obtain: 
$$
\dot \V \leq Md(2\delta-1) <0
$$ 
since $\delta<\frac{1}{2}$. Hence $\V$ converges to $0$ in finite time, and the system reaches consensus, so there exists $T>0$ such that $\iint_{\|x-y\|\leq d} d\mu_x(t) d\mu_y(t) = 1.$ This contradicts the hypothesis \eqref{eq:hypproofbasin}. 
\end{proof}

\subsubsection{Collapse prevention}
Lastly, if the interaction potential is not too big near the origin, we aim to prove that, as in the case of the discrete dynamics, there exists a control keeping the system away from consensus. 

We first show that the control $u$ constructed in Proposition \ref{prop:contVkin} to maximize the time derivative of the variance can maintain the size of the support of $\mu(t)$ above a certain size.

\begin{theorem}
Suppose that  $a\in\Lip(\R^+,\R^+)$, which implies that $\lim_{s\rightarrow 0} sa(s) = 0$. Let $\mu$ be the solution of \eqref{eq:kindym} with control $u$ given by \eqref{eq:contVkin}.
Let $X$ represent the size of $\supp(\mu)$ as defined by \eqref{eq:X}.
Then there exists $\eta>0$ and $\tau>0$ such that for all $t\geq \tau$, $X(t)\geq \eta$.
\end{theorem}
\begin{proof}
From the proof of \ref{th:X}, the evolution of $X^2$ is given by:
$$
\DDt (X(t)^2) = 2\max_{x_0\in K^X} \{ (x(t,x_0)-\bx(t))\cdot \int_{\supp(\mu(t))} a(\|x(t,x_0)-y\|)(y-x(t,x_0)) d\mu(t,y) + \chi_{w(t)} u(t,x(t,x_0)) )\}.
$$
Let $\epsilon<M$, and let $r>0$ be such that for all $s\leq r$, $sa(s)\leq \epsilon$. 
Let $R(c)$ be such that the volume of a ball of radius $R$ is less than $c$. Then we set $w(t)=B(\bx(t),R)\in\Omega_c$.
Suppose $X(t)\leq \min(\frac{r}{2},R)$. Then $\|x-y\|\leq r$ for all $(x,y)\in\supp(\mu)^2$.
Then 
$$
X(t)\dot{X}(t) \geq -\epsilon X(t) + (x(t,x_0)-\bx(t))\cdot \chi_{w(t)} u(t,x(t,x_0)) = (M-\epsilon) X(t).
$$
Hence $\dot X(t)\geq M-\epsilon >0$, so while $X\leq \min(\frac{r}{2},R)$, $X$ increases. Consequently there exists $\tau$ such that for all $t\geq \tau$, $X(t)\geq \frac{1}{2}\min(\frac{r}{2},R)$.
\end{proof}

We proved that if $\lim_{s\rightarrow 0} sa(s) = 0$, there exists a control that keeps the support of $\mu$ from being too small. This implies that with this control, consensus cannot be reached in finite time. However, $\mu$ could still converge to a Dirac mass asymptotically.
%consensus could still be reached asymptotically, even with the support of $\mu$ not shrinking to zero.

In the finite-dimensional system, if $\lim_{s\rightarrow 0} sa(s) = 0$, we can find a control that maintains pairwise distances $\|x_i-x_j\|$ above a certain positive threshold (see Theorem \ref{th:clusteringprevention}).
Similarly, in the kinetic system, we will prove that if initially the diameter of the support of $\mu_0$ is small enough, and $\mu_0$ is contained in non-overlapping balls of small radii, then there exists a control that keeps $\mu_0$ in its initial support, preventing consensus.

\begin{theorem}\label{th:collapsekin}
Suppose that $a\in\Lip(\R^+,\R^+)$, which implies that $\lim_{s\rightarrow 0} s a(s) = 0$.
Let $N\in\N$, $N\geq 2$, and $r>0$ small enough that $N |B(0,r)|\leq c$. Let $R>2r$ be small enough that for all $s\leq R$, $a(s)(s)\leq\epsilon < M$. Let $\{x_1,\ldots, x_N\} \in (\R^d)^N$ be such that for all $i\neq j$, $2r\leq \|x_i-x_j\|\leq \frac{R}{2}$.
If $\mu_0$ satisfies: 
\begin{equation}\label{eq:suppmu0collapse}
\supp(\mu_0)\subset \bigcup_{i=1}^N B(x_i, r),
\end{equation}
then for $\omega:=\bigcup_{i=1}^N B(x_i, r)\subset\Omega_c$ and $u\in\UMkin$ defined by
\begin{equation}\label{eq:contsuppmu0collapse}
u(t,x) = 
\begin{cases}
-M\frac{x-x_i}{\|x-x_i\|} \text{ for all } x\in B(x_i,r)\setminus \{x_i\}, \text{ for all } i\in\{1,\ldots,N\} \\
0 \text{ otherwise,}
\end{cases}
\end{equation} 
the solution to the controlled kinetic equation \eqref{eq:kindym} satisfies for all $t\geq 0$:
$$\supp(\mu(t,\cdot))\subset \bigcup_{i=1}^N B(x_i, r).$$ 
\end{theorem}
\begin{proof}
The proof follows the same argument as that of Theorem \ref{th:safetykin}.
Suppose \eqref{eq:suppmu0collapse} 
with $N$, $r$, $R$ and $\{x_1,\ldots, x_N\} \in (\R^d)^N$ satisfying the conditions listed in the Theorem.
As in the proofs of Theorems \ref{th:X} and \ref{th:safetykin}, we denote by $x(t,x_0)$ the particle trajectory of \eqref{eq:kinsystnocontrol} with $x(0,x_0)=x_0$.
Let $X_i(t) := \max\{\|x(t,x_0)-x_i\|\; , \; x_0\in B(x_i,r)\cap \supp(\mu_0)\}$.
Let $K^i_t = \argmax\{\|x(t,x_0)-x_i\|\; | \; x_0\in B(x_i,r)\cap\supp(\mu_0)\}$.
Then for all $t\geq 0$, for all $x_0\in K^i_t$, $X_i(t)^2 = \| x(t,x_0)-x_i\|^2$.
%From Danskin's theorem, 
%$$
%\DDt(X_i(t)^2) = \max_{x_0\in K_t^i} \{ \DDt \|x(t,x_0)-x_i\|^2\} 
%$$
We have
\begin{equation*}
\begin{split}
\DDt  \|x(t,x_0)-x_i\|^2 = & 2(x(t,x_0)-x_i)\cdot \left(\int_{\supp(\mu(t,\cdot))} a(\|x(t,x_0)-y\|)(y-x(t,x_0)) d\mu_y + \chi_\omega u(t,x(t,x_0))\right) .
\end{split}
\end{equation*}
While $\supp(\mu(t,\cdot))\subset \bigcup_{i=1}^N B(x_i, r)$, for all $(x,y)\in\supp(\mu(t,\cdot))^2$, $\|x-y\|\leq R$, so $a(\|x-y\|)(y-x)\leq \epsilon$.
Then 
$$
\DDt  \|x(t,x_0)-x_i\|^2 \leq 2 (\|x(t,x_0)-x_i\| \epsilon + (x(t,x_0)-x_i)\chi_\omega u(t,x(t,x_0)) ).
$$
As for Theorem \ref{th:safetykin}, we design the control $u$ by \eqref{eq:contsuppmu0collapse}
and the control set by $\omega:=\bigcup_{i=1}^N B(x_i, r)\subset\Omega_c$.
With this control, 
$$
\DDt(X_i(t)^2) = 2 X_i(t) \dot X_i(t) \leq 2 ( \epsilon - M ) X_i(t) <0.
$$
With the designed control, $\mu$ satisfies: $\supp(\mu(t,\cdot))\subset \bigcup_{i=1}^N B(x_i, r)$ for all $t\geq 0$, and consensus is avoided.
\end{proof}

Theorem \ref{th:collapsekin} is a direct adaptation to the kinetic setting of Theorem \ref{th:consensusprevention}, its discrete counterpart. 
Notice that in the discrete case, avoidance of clustering comes as a direct consequence of Theorem \ref{th:consensusprevention}. In the kinetic case, Theorem \ref{th:collapsekin} ensures that the measure stays confined to balls that are disjoint from one another. This only ensures the prevention of consensus, due to the fact that we define kinetic clustering as the presence of one or more Dirac masses (see Definitions \ref{def:consensuskin} and \ref{def:clusterkin}). 
Notice however that here, the assumptions on $a(\cdot)$ are stronger than in the microscopic case: in order to define particle trajectories, we require $a\in\Lip(\R^+,\R^+)$, which would prevent a finite-time blow-up of the solution.

\section{Simulations}\label{Sec:simu}

We illustrate the results proven in the previous sections with numerical simulations, focusing on the finite dimensional model.

\subsection{Black hole and safety region}
An interesting consequence of the results of Sections \ref{Sec:BlackHole} and \ref{Sec:Safety} is the possible coexistence of two regions of the $Nd$-dimensional space of initial configurations of a \textit{black hole} and a \textit{safety region}. We define the \textit{black hole} as the set of initial conditions for which the system tends to the clustering set in finite time. The \textit{safety region} indicates the set of initial conditions for which there exists a control keeping the system away from the clustering set. 

As an illustration of Sections \ref{Sec:BlackHole} and \ref{Sec:Safety}, we consider the interaction function given by: 
$a:s\mapsto \frac{1}{s^2}$. Then indeed $sa(s) = \frac{1}{s}$, so that $\lim_{s\rightarrow 0} sa(s) = +\infty$ and $\lim_{s\rightarrow+\infty} sa(s) = 0$.
This implies the existence of a \textit{black hole} and of a \textit{safety region}. 
We
% fix the bound on the control $M=1$ and 
 study the geometry of these regions. 
Let $g:s\rightarrow - \frac{1}{s}$ define the generalized entropy $W_g$. 
Then from the proof of Theorem \ref{th:BH}, we know that if $V(0)\leq \frac{\epsilon^2}{\sqrt{2}N}$, then the system converges to consensus in finite time, where $\epsilon$ is such that for any $A>M$, if $s\leq \epsilon$, then $sa(s)\geq A$. Let $\delta>0$ arbitrarily close to 0. Then $\epsilon = \frac{1}{M+\delta}$ satisfies the condition. Hence the Black Hole region $\mathcal{R}^M_{\mathrm{BH}}$ satisfies:
$$
 \{ x\in\RR^{dN}\; | \; \sum\limits_{i<j} \|x_i-x_j\|^2 < \frac{1}{M^2} \} \subseteq \mathcal{R}^M_{\mathrm{BH}}.
$$

Similarly, from the proof of Theorem \ref{th:safety1}, we know that if $W_g(0)\geq \frac{1}{2N^2}(g(\mu^2)+m(\frac{N(N-1)}{2}-1))$, then with the proper control, the system stays bounded away from the clustering set, where $\mu$ is such that for any $\epsilon<\frac{M}{N}$, if $s\geq\mu$, then $sa(s)\geq\epsilon$. Then for any $\delta>0$ arbitrarily close to 0, let $\epsilon = \frac{M}{N}-\delta$. With the function $a:s\mapsto \frac{1}{s^2}$, the condition above is satisfied for $\mu = \frac{1}{\epsilon}$. Hence with the choice $g:s\mapsto -\frac{1}{s}$, we have $m=0$ and
the safety region $\mathcal{R}^M_{\mathrm{S}}$ satisfies: 
$$
 \{ x\in\RR^{dN}\; | \; \sum\limits_{i<j} \frac{1}{\|x_i-x_j\|^2} < \frac{M^2}{N^2} \} \subseteq \mathcal{R}^M_{\mathrm{S}}.
$$

\begin{figure}[H]
\centering
\includegraphics[width=0.4\textwidth]{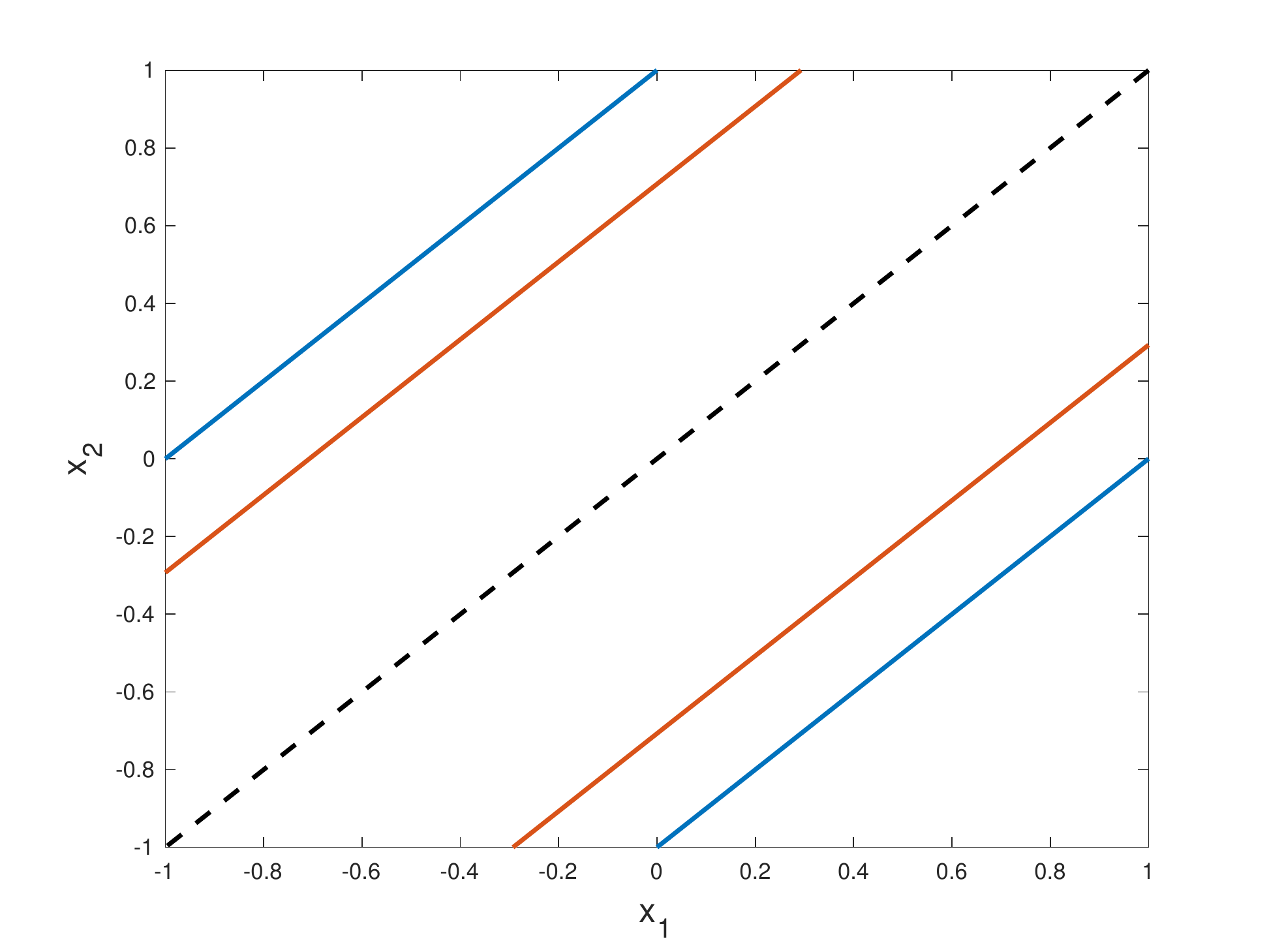}
\includegraphics[width=0.4\textwidth]{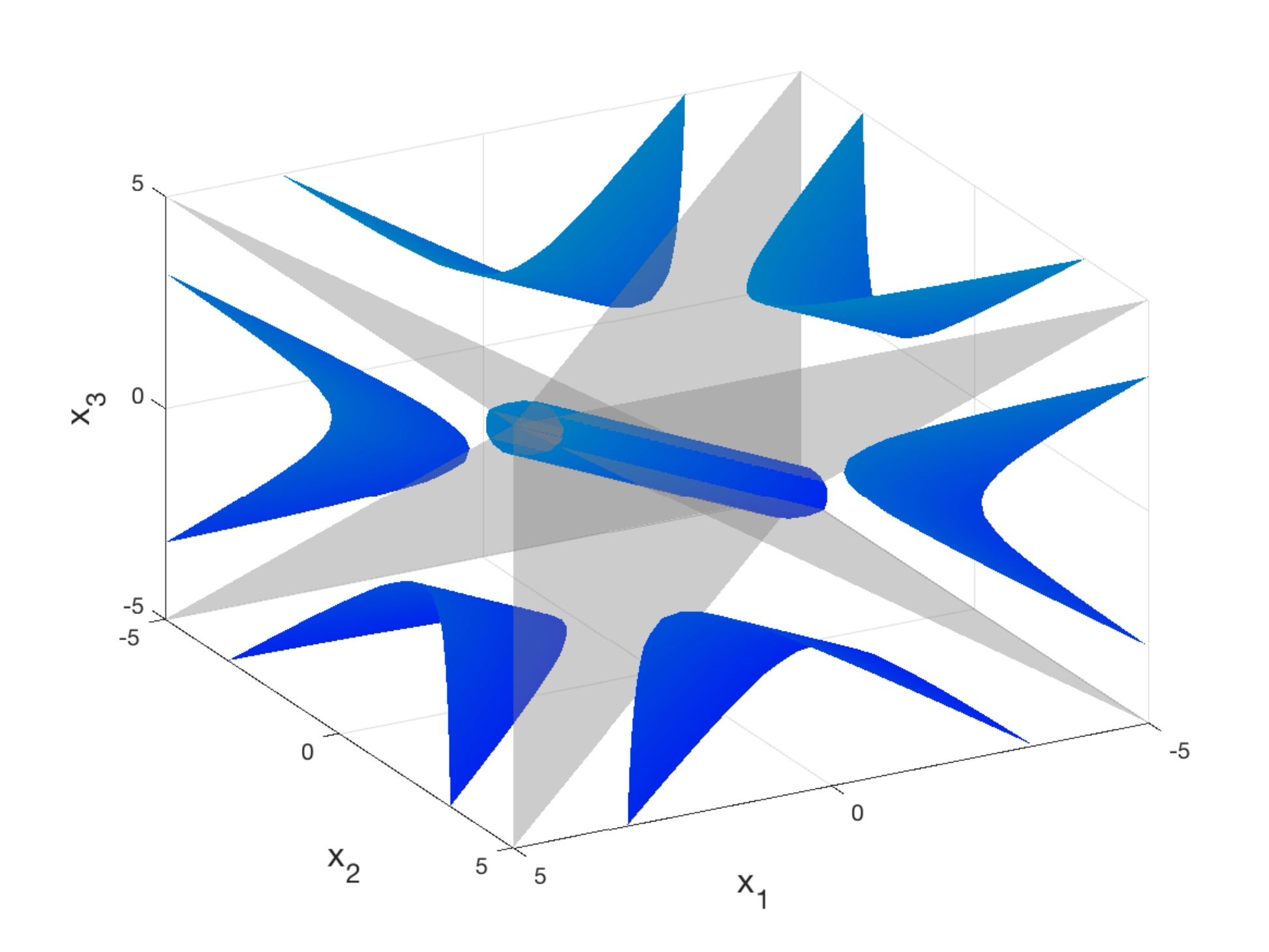}
\caption{Partition of the state space into the\textit{ black hole} region and the \textit{safety region} for $M=1$, $a:s\mapsto\frac{1}{s^2}$ and $g:s\mapsto-\frac{1}{s}$. Left: with $(N,d)=(2,1)$, the region enclosed by the red lines is a subset of $\mathcal{R}^M_{\mathrm{BH}}$ and the region located outside the blue lines is a subset of $\mathcal{R}^M_{\mathrm{S}}$. The dotted line represents the consensus manifold.
Right: with $(N,d)=(3,1)$, the region inside the central cylinder is a subset of $\mathcal{R}^M_{\mathrm{BH}}$ and the region outside the hyperbola branches are a subset of $\mathcal{R}^M_{\mathrm{S}}$. The grey planes represent the clustering set, and their intersection is the consensus manifold.}
\label{fig:BH_SZ}
\end{figure}

%\begin{figure}[H]
%\centering
%\includegraphics[width=0.4\textwidth]{BH_10agents.eps}
%\includegraphics[width=0.4\textwidth]{BH_10agents_Wg.eps}
%\caption{ With initial conditions $x_0\in\mathcal{R}_{\mathrm{BH}}$, the system converges to the clustering set in finite time. The controlled agent is in red.}
%\label{fig:BH_10}
%\end{figure}
%
%\begin{figure}[H]
%\centering
%\includegraphics[width=0.4\textwidth]{SZ_10agents.eps}
%\includegraphics[width=0.4\textwidth]{SZ_10agents_Wg2.eps}
%\caption{ With initial conditions $x_0\in\mathcal{R}_{\mathrm{S}}$, the control $u_W$ manages to steer the system away from the consensus set. The controlled agent is in red.}
%\label{fig:SZ_10}
%\end{figure}

As an illustration of the results proven in Sections \ref{Sec:BlackHole} and \ref{Sec:Safety}, we provide an example of an interaction function allowing for the existence of both a \textit{black hole} and a \textit{safety region}. 
Let $a:s\mapsto\frac{1}{s^2}$. Then $\lim_{s\rightarrow 0} sa(s) = +\infty$ and $\lim_{s\rightarrow +\infty} sa(s) = 0$, which means that there exist a \textit{black hole} and a \textit{safety region}. 
In the following, we will use the generalized entropy $W_g$ defined with $g:s\mapsto -\frac{1}{s}$.
\begin{itemize}
\item From the proof of Theorem \ref{th:BH}, $x_0\in\mathcal{R}^M_\mathrm{BH}$ if $ V(0) < \frac{1}{2N^2M^2}$, \textit{i.e.}, if 
 $M < \frac{1}{N\sqrt{2V(0)}}.$
\item From the proof of Theorem \ref{th:safety1}, $x_0\in\mathcal{R}^M_\mathrm{S}$ if $ W_g(0) > -\frac{M^2}{2N^4}$, \textit{i.e.}, if 
 $M > N^2\sqrt{-2W_g(0)}. $
\end{itemize}

Figures \ref{fig:BH_10} 	and \ref{fig:SZ_10} illustrate these results with $N=10$. The initial positions of the agents were distributed randomly and gave $W_g=-0.207$.

\begin{figure}[H]
\centering
\includegraphics[width=0.4\textwidth]{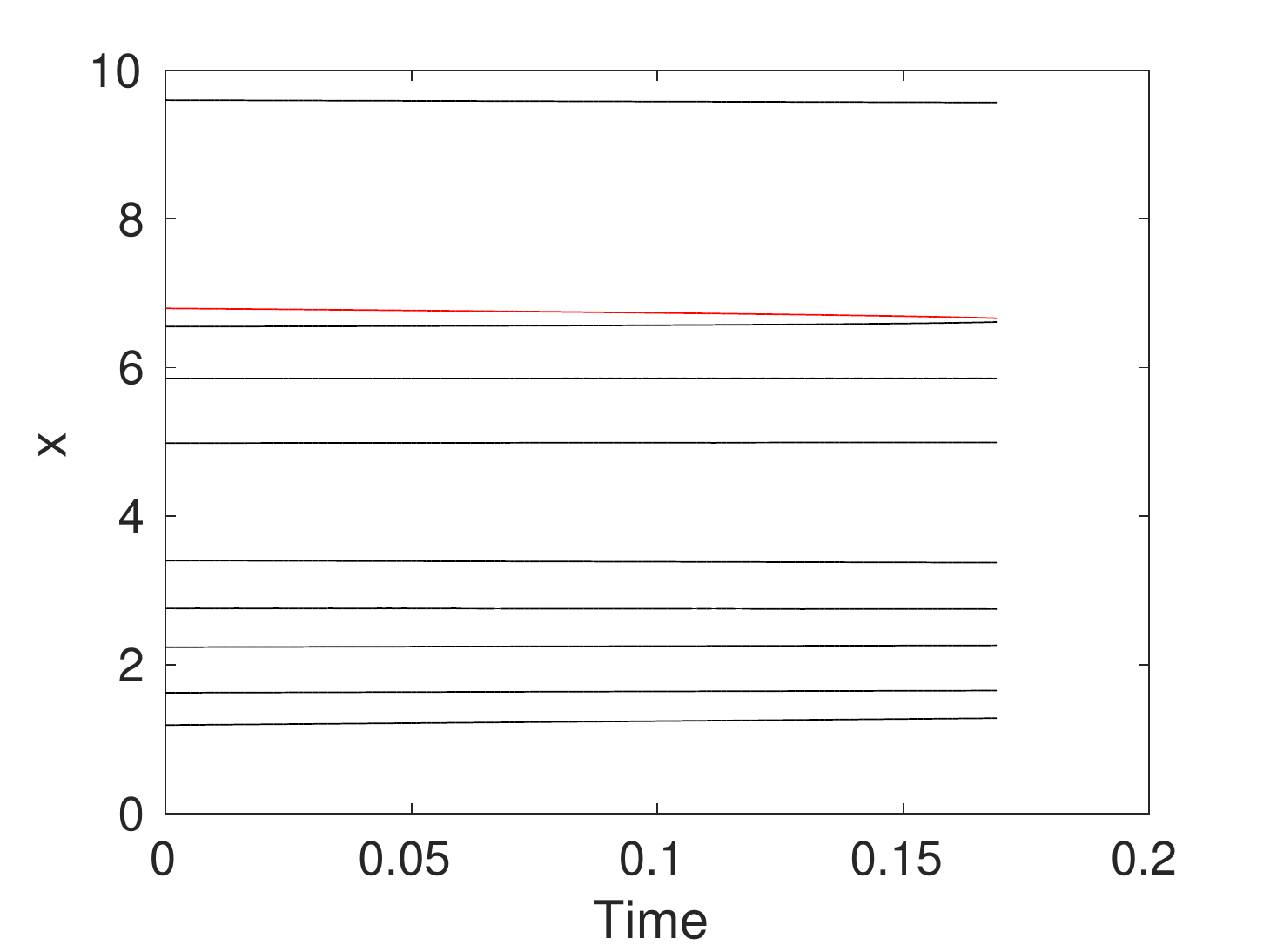}
\includegraphics[width=0.4\textwidth]{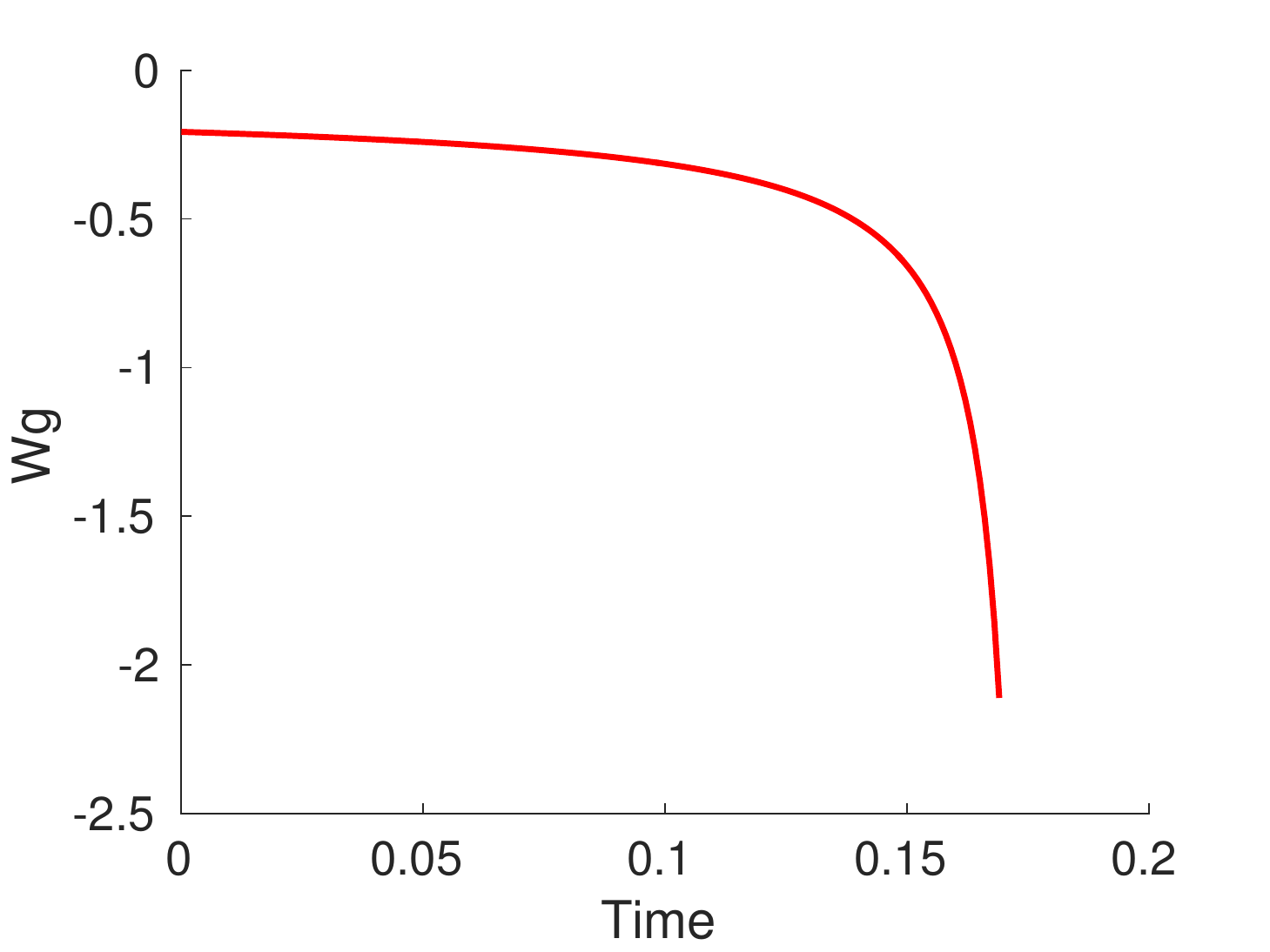}
\caption{ With $M = 0.16 < \frac{1}{N\sqrt{2V(0)}}$, the system satisfies $x_0\in\mathcal{R}^M_{\mathrm{BH}}$, and converges to the clustering set in finite time. Left: Evolution of the 10 agents' positions (the controlled agent is in red). Right: Evolution of the generalized entropy.}
\label{fig:BH_10}
\end{figure}

\begin{figure}[H]
\centering
\includegraphics[width=0.4\textwidth]{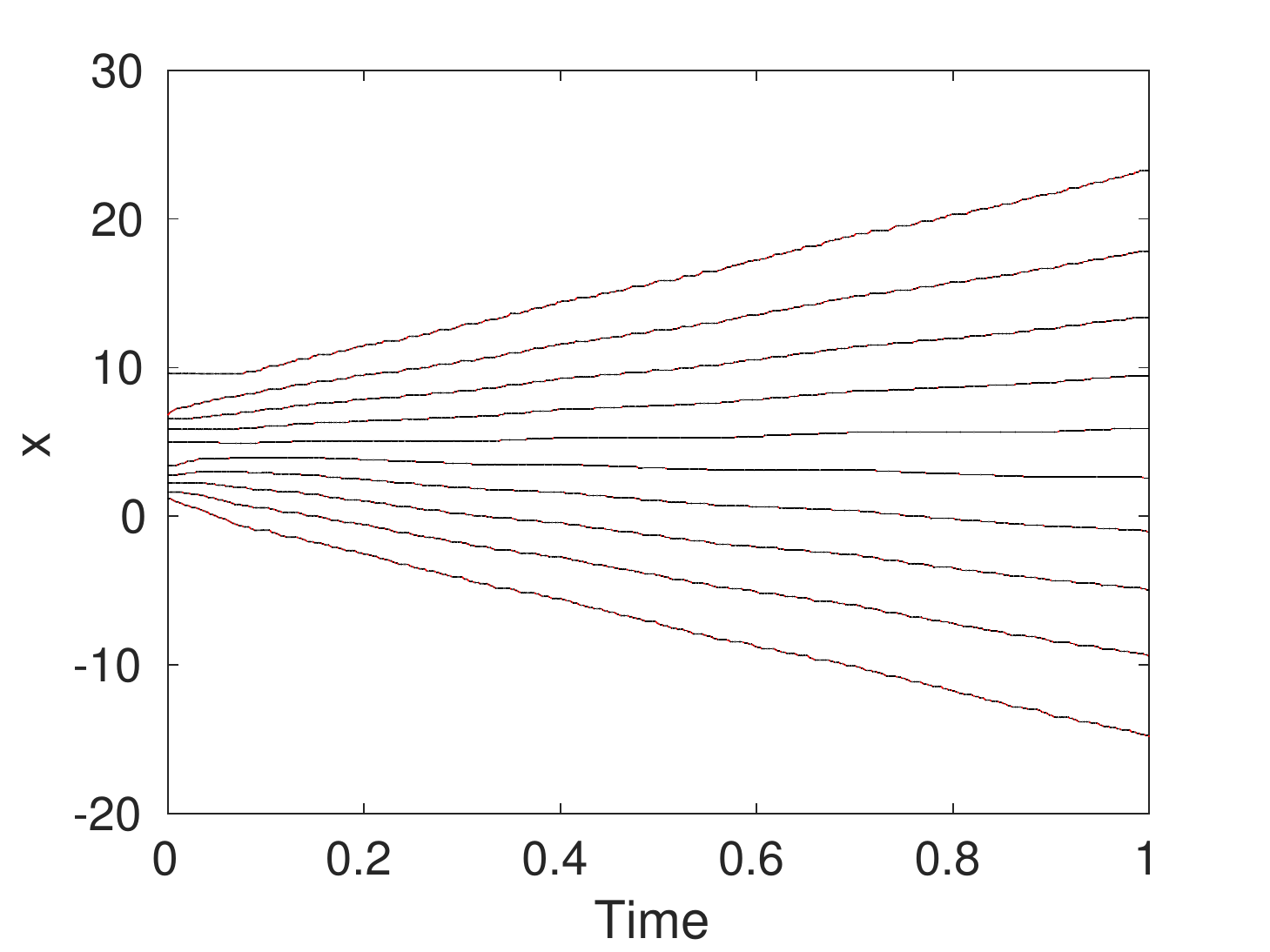}
\includegraphics[width=0.4\textwidth]{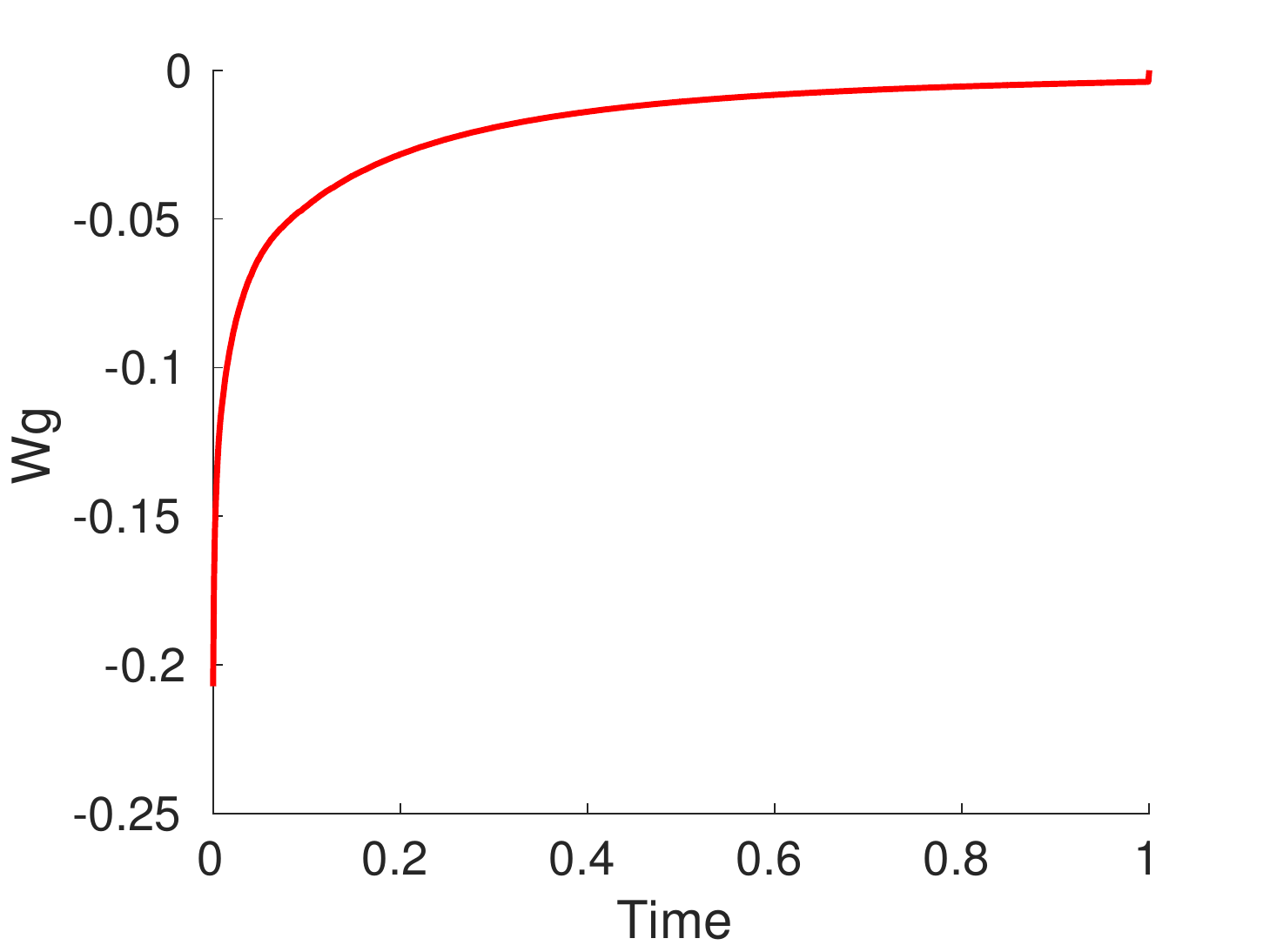}
\caption{ With $M = 0.16 > N^2\sqrt{-2W_g(0)}$, the system satisfies $x_0\in\mathcal{R}^M_{\mathrm{S}}$, and the control $u_W$ manages to steer the system away from the consensus set. Left: Evolution of the 10 agents' positions (the controlled agent is in red). Right: Evolution of the generalized entropy.}
\label{fig:SZ_10}
\end{figure}

\subsection{Basin of attraction and collapse prevention} 

To illustrate the cases of Sections \ref{Sec:basin} and \ref{Sec:CollapsePrev}, we now consider the interaction function
$a:s\mapsto\frac{1}{\sqrt{s}}$. Notice that $a$ satisfies the condition for the existence of a \textit{basin of attraction} ( $\lim\limits_{s\rightarrow +\infty} sa(s) = +\infty$) and the condition for the possibility of \textit{collapse prevention} ($\lim\limits_{s\rightarrow 0} sa(s) = 0$). Let $N=10$ and $M=1$. 
\begin{itemize}
\item By Theorem \ref{th:basin}, there exists $T>0$ such that 
$
x(T)\in B:= \{(x_i)_{i\in\{1,\ldots,N\}} \; | \min\limits_{i\neq j} \|x_i-x_j\| \leq 1\}.
$
\item By Theorem \ref{th:clusteringprevention}, there exists $\kappa>0$ such that 
$
\forall t\geq 0, \; \forall (i,j)\in\{1,\ldots,N\}^2, \; \|x_i(t)-x_j(t)\| \geq \kappa.
$
\end{itemize}

Figure \ref{fig:BA_10} shows the evolution of a system initially outside of the \textit{basin of attraction} $B$, \textit{i.e.}, $\|x_i(0)-x_j(0)\|>1$ for all $(i,j)\in\{1,\ldots,N\}^2$ (with $W_g(0)=-0.2$). Despite being controlled with the strategy $u_W$ defined in Section \ref{Sec:Entropy}, the system converges to $B$. However, notice that the generalized entropy does not tend to $-\infty$, as the system is able to prevent convergence to the clustering set. 
Figure \ref{fig:CP_10} shows the evolution of a system initially very close to the clustering set ($W_g(0) = -7.2$). The control strategy defined in Theorem \ref{th:clusteringprevention} is successful in preventing clustering.

\begin{figure}[H]
\centering
\includegraphics[width=0.4\textwidth]{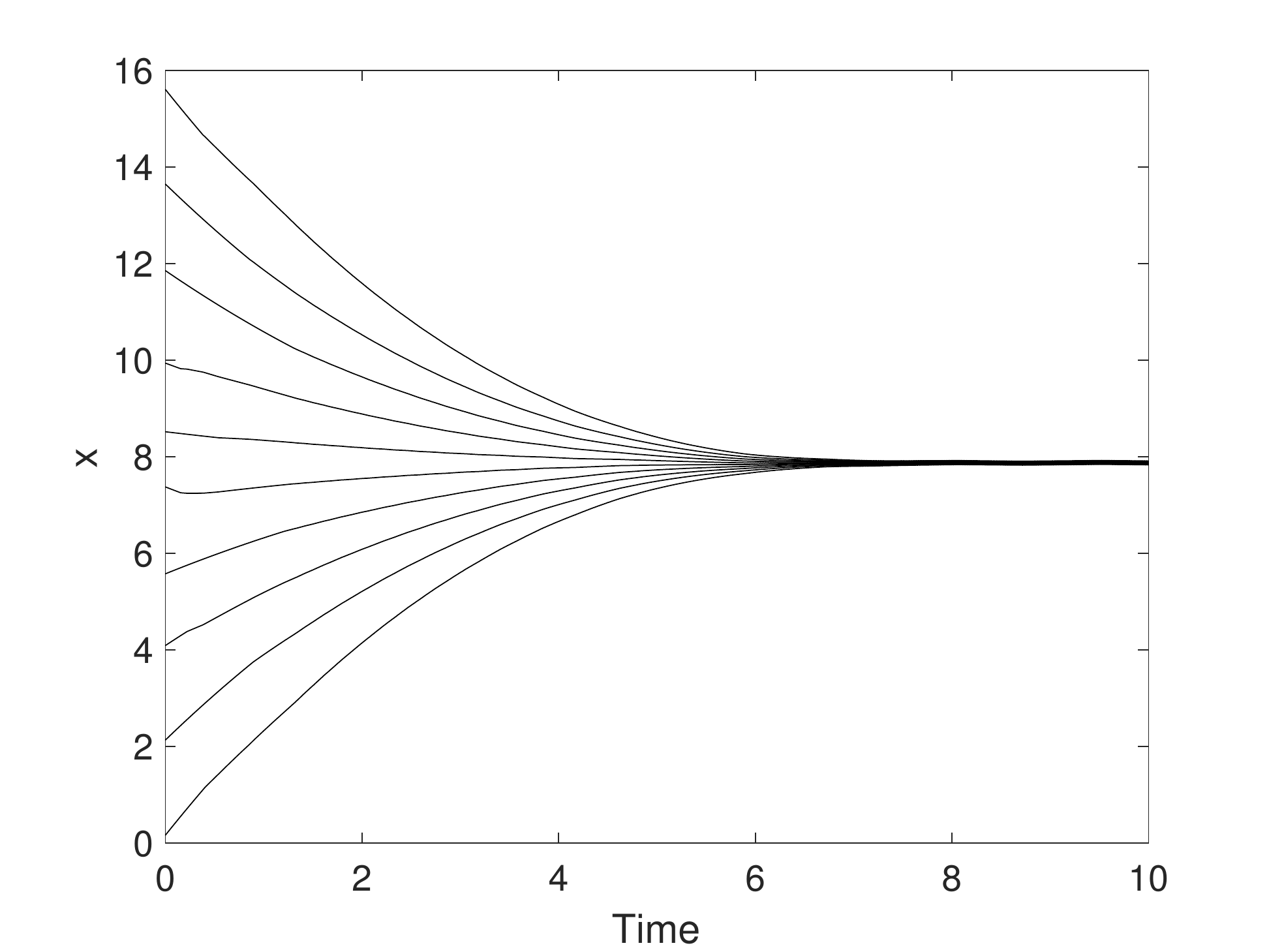}
\includegraphics[width=0.4\textwidth]{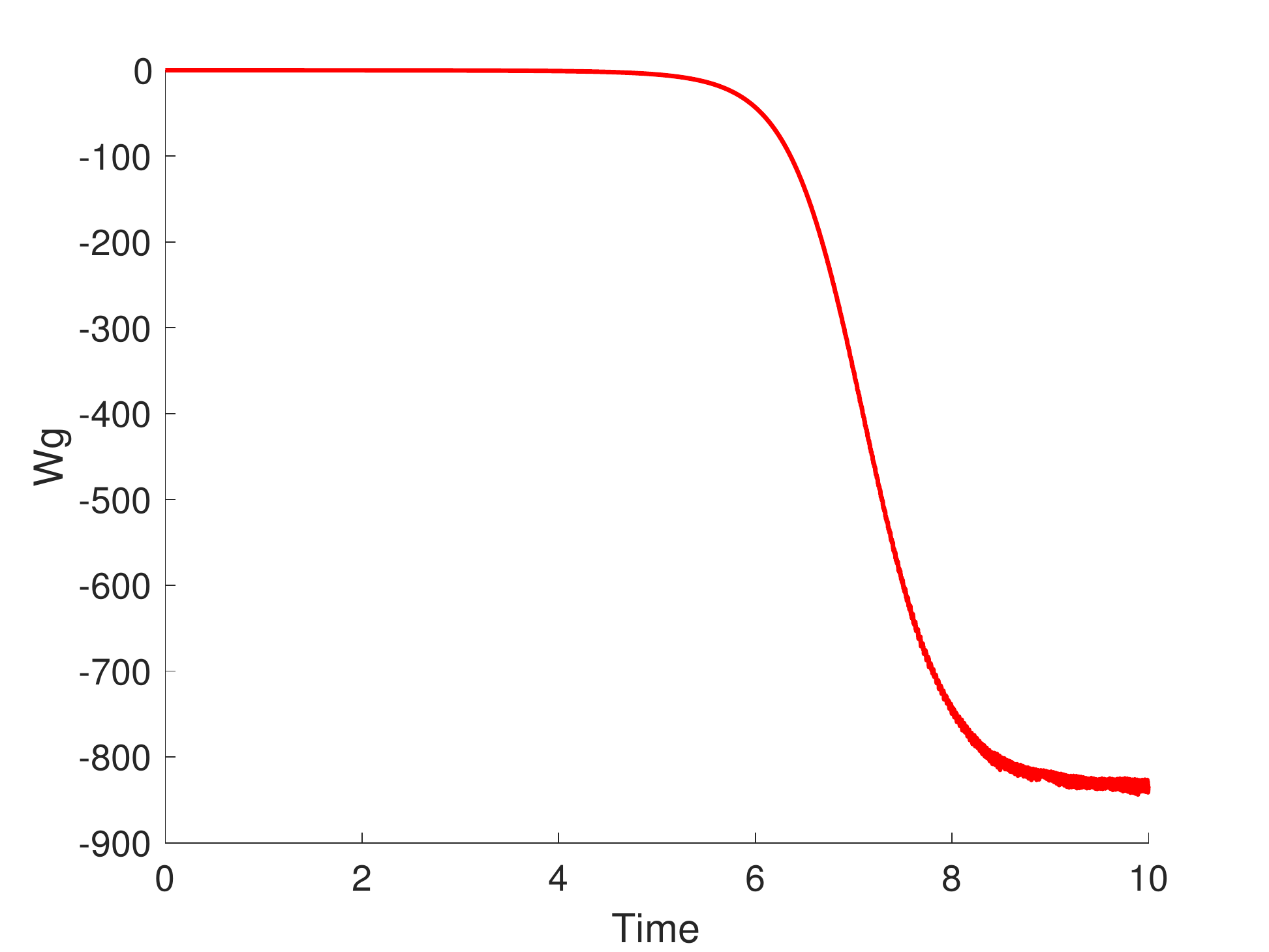}
\caption{  No control can prevent the convergence of the system to the \textit{basin of attraction} $B = \{(x_i)_{i\in\{1,\ldots,N\}} \; | \min\limits_{i\neq j} \|x_i-x_j\| \leq 1\}$.
Left: Evolution of the 10 agents' positions. Right: Evolution of the generalized entropy.}
\label{fig:BA_10}
\end{figure}

\begin{figure}[H]
\centering
\includegraphics[width=0.4\textwidth]{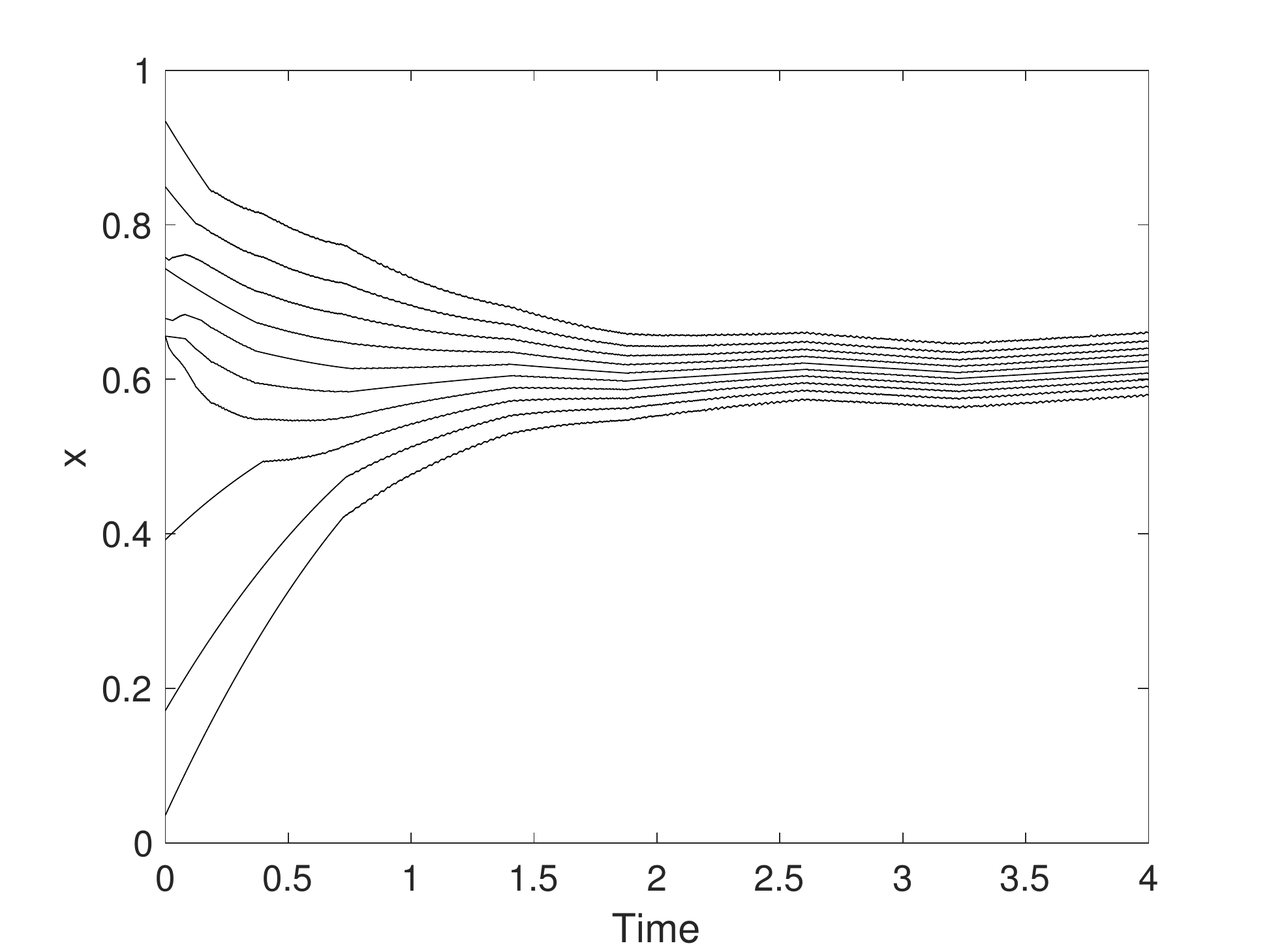}
\includegraphics[width=0.4\textwidth]{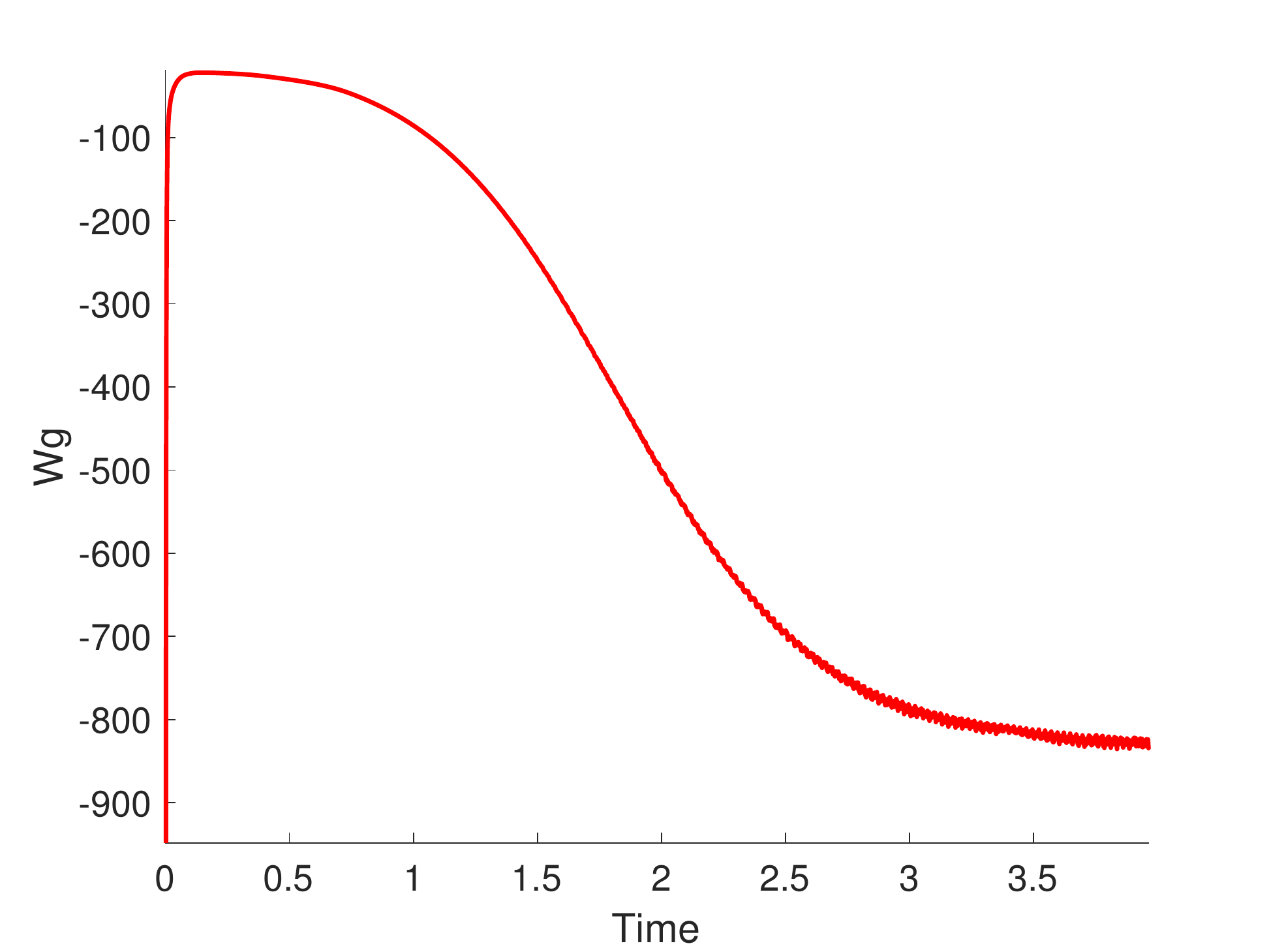}
\caption{ The control $u_W$ steers the system away from the clustering set. Left: Evolution of the 10 agents' positions. Right: Evolution of the generalized entropy.}
\label{fig:CP_10}
\end{figure}

\section*{Conclusion and further comments}

The problem of controlling collective dynamics systems to achieve consensus or alignment has been frequently studied. In this paper, we focus on the opposite problem: controlling a system that naturally converges to consensus to achieve \textit{declustering}. 
We first remark that the standard variance used to characterize consensus (with the equivalence $V=0 \Leftrightarrow \forall (i,j)\in\{1,\ldots N\}^2,\; x_i=x_j$) does not measure declustering. Instead, to characterize the state of declustering, we introduce a generalized entropy functional $W_g$ with the property: $W_g > c \Leftrightarrow \forall (i,j)\in\{1,\ldots N\}^2, \; x_i\neq x_j$.
With this tool in hand, we design a control strategy aiming to prevent the system from clustering  by maximizing instantaneously the time derivative of $W_g$. The control thus constructed is \textit{sparse}, meaning that it only acts on one agent at a time. 
As opposed to the problem of achieving consensus, here we fight against the system's natural tendency to form clusters. For this reason we do not expect to succeed in every situation.  
Indeed, the analysis of the first-order opinion formation model \eqref{eq:Krause} with a positive interaction function $a(\cdot)$ and additive control reveals the existence of four regions of $(\R^d)^N$ that determine whether the system can be maintained away from consensus or clustering (see Table \ref{table:summary}). The behavior of $a(\cdot)$ near zero determines the existence either of a black hole region or of the possibility to prevent collapse of the system. 
In the black hole region, no control can keep the system away from consensus, whereas in the case of collapse prevention there exist controls that can keep the system from clustering. The black hole region and the collapse prevention region can coexist with either a safety region or a basin of attraction, which are determined by the behavior of $a(\cdot)$ at infinity. In the safety region, the system can be kept far from the clustering set, given suitable initial conditions. On the opposite, in the case of a basin of attraction, the system is attracted to a neighborhood of the clustering set. 

As seen in Section \ref{Sec:kinetic}, most results for the microscopic model can be extended to the kinetic equation \eqref{eq:kindym}. 
We define kinetic clustering as the presence of one or more Dirac masses in the population density, and we show that declustering can be characterized by the kinetic version of the generalized entropy. Similarly to the microscopic case, we design sparse control strategies maximizing the time derivative of the kinetic generalized entropy instantaneously. 
To extend the notion of sparse control to the PDE framework, we set a bound on the size of the support of the control. 
As in the microscopic setting, we show the existence of the four zones determined by the behavior of $a(\cdot)$ at zero (the black hole and the collapse prevention zones) and at infinity (the safety zone or the basin of attraction).

Lastly, in Section \ref{Sec:simu}, we present numerical simulations illustrating those four situations in the microscopic case, with two examples of interaction functions. With $a:s\mapsto s^{-2}$, we observe the coexistence of a black hole and of a safety region. Given the same initial conditions, the convergence to consensus or the avoidance of clustering are determined by the allowed strength of the control $M$.
With the interaction function $a:s\mapsto s^{-1/2}$ and a fixed bound on the control $M=1$, if it is initially far from the clustering set, the system converges to a basin of attraction. However, if the initial conditions are already in a neighborhood of the clustering set, we show that collapse to clustering can be avoided. 
\\
%Interestingly, due to the  fundamentally different natures of the finite-dimensional and kinetic formulations, when extending the property of avoidance of clustering to infinite dimension, one only obtains the avoidance of consensus, a weaker property (see Theorems \ref{th:safetykin} and \ref{th:collapsekin}). 
%Numerical simulations illustrate the results found for the finite-dimensional system. 

This work can be extended in many ways. We list a few of the possible future directions that stem naturally from the results presented above.
\paragraph{Sparse kinetic control.} The crucial notion of \textit{sparse} control in finite-dimension can be extended to the kinetic setting in various ways. In this work, we made the choice of designing controls of the form $u\chi_\omega$, where we not only bound the $L^\infty$ norm of $u$, but also the size of the controlled region $\omega$, with the condition $\int_\omega dx \leq c$ for some positive constant $c$. Another way to impose a notion of kinetic sparsity would be to bound the size of the controlled population, with the condition $\int_\omega d\mu(x) \leq c$, as done for example in \cite{PRT15}. 
In the context of clustering prevention, this condition is more restrictive, as we want to act mainly on the concentrated part of the population. One could investigating whether the results from Section \ref{Sec:kinetic} hold for controls that satisfy $\int_\omega d\mu(x) \leq c$.
\paragraph{Black hole horizon.} The analysis of the model in Section \ref{Sec:finite} and the numerical simulations in Section \ref{Sec:simu} show the possible coexistence of a black hole and a safety region. One could go further in investigating the nature of the boundary between the two regions, and the existence of the \textit{black hole horizon}. Two scenarios can be anticipated: either the black hole and the safety region form all of the state space, with $(\R^d)^N = \mathcal{R}^M_\mathrm{BH}\cup \mathcal{R}^M_\mathrm{S}$, or there exists a black hole horizon, \textit{i.e.} a region of the state space that is neither a black hole nor a safety region $\mathcal{H}^M_{\mathrm{BH}} = (\R^d)^N \setminus (\mathcal{R}^M_\mathrm{BH}\cup \mathcal{R}^M_\mathrm{S})$.
\paragraph{Optimal control in finite dimension.} The sparse controls designed to prevent the system from clustering minimize instantaneously the generalized entropy $W_g$. One could instead look for global minimizers of the functional $W_g$ in order to design optimal control strategies. 
\paragraph{Second-order clustering.}
A natural future direction for this work would consist in investigating second-order clustering, in a model inspired by that of Cucker and Smale \cite{CS07}. 
Second-order models are commonly applied to animal groups to study coordinated
collective behavior \cite{ACMPPRT17, L13}. 
In this framework, the variable of interest is the velocity, and agreement of all agents' velocities is referred to as ``alignment'' \cite{CFPT13, CFPT15}. It follows quite naturally to define clustering as agreement of several agents' velocities, and one could conduct a similar analysis of the resulting model, defining the second-order generalized entropy using the velocity variables.
\paragraph{Well-posedness of the kinetic aggregation equation.}
The well-posedness of the kinetic equation \eqref{eq:kindym} with aggregation phenomena is currently being investigated by several groups (see for instance \cite{BLR11} and \cite{CCR11} and references within). 
The existence and uniqueness of a solution to \eqref{eq:kindym} for singular interaction potentials is a highly non-trivial problem. Intuitively, highly attractive interaction functions can give rise to a finite-time blow-up of the solution. This blow-up need not happen instantaneously and one could observe the coexistence of one or several Dirac masses with absolutely continuous parts of the measure in between them. A possible approach to define weak measure solutions would be through measure differential equations, which would allow to prolong solutions past blow-up times. One could then go further and define control in a new framework, that of measure differential equations \cite{P17}.

\newpage 
\vfill
\newpage

%%%%%%%%%%%%%%%%%%%%%%%%%%%%%%%%%%%%%%%%%%%%%%%%%%%%%%%%%%%%%%%%%%%%%%%%%%%%%%%%%%%%%
%%% Bibliography %%%%%%%%%%%%%%%%%%%%%%%%%%%%%%%%%%%%%%%%%%%%%%%%%%%%%%%%%%%%%%%%%%%%
%%%%%%%%%%%%%%%%%%%%%%%%%%%%%%%%%%%%%%%%%%%%%%%%%%%%%%%%%%%%%%%%%%%%%%%%%%%%%%%%%%%%%

\bibliography{BlackSwan-Biblio}{}
\bibliographystyle{abbrv}

\end{document}